\documentclass[a4paper,10pt]{amsart}

\usepackage{amsmath,amsthm,amssymb,amsfonts,float,enumerate,hyperref}
\usepackage[pdftex]{graphicx}

\allowdisplaybreaks[3]

\newcommand{\R}{\mathbb{R}}

\newcommand{\Z}{\mathbb{Z}}
\newcommand{\N}{\mathbb{N}}

\newcommand{\PV}{\operatorname{P.V.}}

\renewcommand{\Re}{\operatorname{Re}}

\newtheorem{thm}{Theorem}[section]

\newtheorem{lem}[thm]{Lemma}

\theoremstyle{definition}
\newtheorem{defn}[thm]{Definition}
\newtheorem{rem}[thm]{Remark}

\numberwithin{equation}{section}



\begin{document}

\title[Fractional discrete Laplacian vs discretized fractional Laplacian]
{Fractional discrete Laplacian\\versus discretized fractional Laplacian}

\author[\'O. Ciaurri, L. Roncal, P. R. Stinga, J. L. Torrea, and J. L. Varona]
{\'Oscar Ciaurri \and Luz Roncal \and Pablo Ra\'ul Stinga \and Jos\'e L. Torrea \and Juan Luis Varona}

\address[\'O. Ciaurri, L. Roncal, and J. L. Varona]{%
Departamento de Matem\'aticas y Computaci\'on\\
Universidad de La Rioja\\
26004 Logro\~no, Spain}
\email{\{oscar.ciaurri,luz.roncal,jvarona\}@unirioja.es}

\address[P. R. Stinga]{Department of Mathematics\\
The University of Texas at Austin\\
1 University Station, C1200, Austin\\
TX 78712-1202, United States of America}
\email{stinga@math.utexas.edu}

\address[J. L. Torrea]{Departamento de Matem\'aticas, Facultad de Ciencias\\
Universidad Aut\'onoma de Madrid\\
28049 Madrid, Spain\\
and
ICMAT-CSIC-UAM-UCM-UC3M}
\email{joseluis.torrea@uam.es}

\thanks{Research partially supported by grants MTM2011-28149-C02-01
and MTM2012-36732-C03-02 from Spanish Government}

\begin{abstract}
We define and study some properties of the fractional powers of the discrete Laplacian
$$(-\Delta_h)^s,\quad\hbox{on}~\mathbb{Z}_h = h\mathbb{Z},$$
for $h>0$ and $0<s<1$.
A comparison between our fractional discrete Laplacian and the \textit{discretized} continuous fractional Laplacian as $h\to0$
is carried out. We get estimates in $\ell^\infty$ for the error of the approximation
in terms of $h$ under minimal regularity assumptions.
Moreover, we provide a pointwise formula with an explicit kernel and deduce H\"older estimates
for $(-\Delta_h)^s$.
A study of the negative powers (or discrete fractional integral) $(-\Delta_h)^{-s}$ is also sketched.
Our analysis is mainly performed in dimension one. Nevertheless, we show certain asymptotic
estimates for the kernel
in dimension two that can be extended to higher dimensions.
Some examples are plotted to illustrate the comparison in both one and two dimensions.
\end{abstract}

\maketitle

\section{Introduction and main results}
\label{Intro}

The fractional Laplacian, understood as a positive power of the classical Laplacian,
has been present for long time in several areas of Mathematics, like Fractional Calculus and Functional Analysis.
However, although this operator was used for some differential equations in Physics, 
it was not until the first part of the last decade when it became a very popular object in
the field of Partial Differential Equations.  Indeed, equations involving fractional Laplacians have been
one of the most studied research topics in the present century.
Nowadays hundreds of papers can be found in the literature with the phrase ``fractional Laplacian'' in their titles.
We recall  that the fractional  Laplacian on $\R^n$ can be defined
for $0<s<1$ and good enough functions $u$ as, see~\cite{Landkof},
\[
(-\Delta)^su(x) =  c_{n,s}\PV \int_{\R^n} \frac{u(x)-u(y)}{|x-y|^{n+2s}}\,dy.
\]
The trigger that produced the outbreak in the field
was the paper by L. Caffarelli and L.~Silvestre \cite{Caffarelli-Silvestre}. The essence of that work is to immerse the fractional Laplacian (a nonlocal operator)  on $\R^n$ into a (local) partial differential equation problem in $\R^{n+1}$, where the known methods of PDEs can be applied. Since the appearance of that paper there has been
a substantial revision of a big amount of problems in differential equations where the 
Laplacian is replaced by the fractional Laplacian.
Of course in this trend the Numerical Analysis should have to be present.
Several recent works can be found in this direction, see \cite{Bonito-Pasciak,DelTeso,Huang-Oberman,Nochetto,Zoia}.
One of the main troubles to overcome in the numerical approach is the nonlocality of the operator.

On the other hand, in the past few years the language of semigroups has been used
as a versatile tool in the study
of different aspects and scenarios for the fractional Laplacian
with a great success. See \cite{Caffarelli-Stinga,FiveGuys,Gale, Roncal-Stinga, RS, Stinga, Stinga-Torrea}. 
In particular, in \cite{FiveGuys} it was possible to define the fractional discrete Laplacian 
applying this approach.
Therefore the semigroup language gives us the chance to define and study the fractional discrete Laplacian.
An obvious question arises immediately: can this language be used to develop some
kind of Numerical Analysis for the fractional Laplacian? This is exactly our aim here, to present a discretization method to approximate the fractional Laplacian.

Let us give a brief description of our results.
Consider a mesh of size $h>0$ on $\R$ given by
$\Z_h = \{ hj:j\in\Z\}$. The discrete Laplacian $\Delta_h$ on the mesh is given by
$$-\Delta_hu_j = -\frac{1}{h^2} (u_{j+1}-2u_{j}+u_{j-1}).$$
Here $u_j$, $j\in\Z$, is a function on $\Z_h$. 
The \textit{fractional discrete Laplacian} $(-\Delta_h)^s$ is defined with 
the semigroup language:
\begin{equation}
\label{definition}
(-\Delta_h)^su_j=\frac{1}{\Gamma(-s)}\int_0^\infty\big(e^{t\Delta_h}u_j-u_j\big)\,\frac{dt}{t^{1+s}},
\end{equation}
where $w_j(t)=e^{t\Delta_h}u_j$ is the solution  to the semidiscrete heat equation
\begin{equation*}
\begin{cases}
\partial_tw_j=\Delta_hw_j, &\text{in}~\Z_h\times(0,\infty),\\
w_j(0)=u_j, &\text{on}~\Z_h.
\end{cases}
\end{equation*}
Now given a function $u=u(x)$ defined on $\R$, we consider its restriction $r_hu$
to the mesh as defined by $(r_h u)_j\equiv r_hu(j):= u(hj)$.
Our goal is to estimate, as a function of the size $h$ of the mesh, differences of the type
\[
\big\|(-\Delta_h)^s(r_hu)-r_h \big((-\Delta)^s u\big) \big\|_{\ell^\infty}.
\]
Certainly the estimates will depend on the smoothness of the function,
as our first main result shows.

\begin{thm}
\label{thm:consistencia1d}
Let $0<\beta\le 1$ and $0<s<1$.
\begin{enumerate}[(i)]
\item \label{(iCon)} Let $u\in C^{0,\beta}(\R)$ and $2s<\beta$. Then
$$\|\, (-\Delta_h)^s(r_hu)-r_h((-\Delta)^su)\, \|_{\ell^\infty}\leq C[u]_{C^{0,\beta}(\R)}h^{\beta-2s}.$$
\item \label{(iiCon)} Let $u\in C^{1,\beta}(\R)$ and $2s<\beta$. Then
$$\|D_+(-\Delta_h)^s(r_hu)-r_h(\tfrac{d}{dx}(-\Delta)^su)\|_{\ell^\infty}\leq C[u]_{C^{1,\beta}(\R)}h^{\beta-2s}.$$
\item \label{(iiiCon)} Let $u\in C^{1,\beta}(\R)$ and $2s>\beta$. Then
$$\|(-\Delta_h)^s(r_hu)-r_h((-\Delta)^su)\|_{\ell^\infty}\leq C[u]_{C^{1,\beta}(\R)}h^{\beta-2s+1}.$$
\item \label{(ivCon)} Let $u\in C^{k,\beta}(\R)$ and assume that $k+\beta-2s$ is not an integer, with $2s<k+\beta$. Then
$$\|D_+^l(-\Delta_h)^s(r_hu)-r_h(\tfrac{d^{l}}{dx^{l}}(-\Delta)^su)\|_{\ell^\infty}\leq
C[u]_{C^{k,\beta}(\R)}h^{\beta-2s+k-l},$$
where $l$ is the integer part of $k+\beta-2s$.
\end{enumerate}
The constants $C$ above depend only on $s$ and $\beta$, but not on $h$ or $u$ and
$\frac{d}{dx}$ denotes the standard derivative. For a function $u_j$ on the mesh, we let
$$D_+u_j= \frac{u_{j+1}-u_j}{h},$$
and by $D_{+}^{l}$ we mean that we apply the operator $D_{+}$ to $u_j$ $l$-times.
\end{thm}

To prove this result we need the explicit formula for $(-\Delta_h)^s$ presented in the following statement.

\begin{thm}
\label{thm:basicProperties} 
Let $0\leq s\leq1$ and
$\displaystyle
\ell_{s} := \bigg\{u:\Z_h\to\R \;:\;
\|u\|_{\ell_s}:= \sum_{m\in \Z} \frac{|u_{m}|}{(1+|m|)^{1+2s}}<\infty\bigg\}.
$
\begin{enumerate}[(a)]
\item \label{aa} For $0<s<1$ and $u\in\ell_{s}$ we have
\begin{equation} \label{eq:puntualdiscreta}
(-\Delta_h)^su_{j} = \frac{1}{h^{2s}} \sum_{m\neq j}
\big(u_{j}-u_{m}\big) K^s(j-m),
\end{equation}
where the discrete kernel  $K_s$ is given by
\begin{equation}
\label{eq:fractionalKernelOned}
K_s(m) = \frac{4^s\Gamma(1/2+s)\Gamma(|m|-s)}{\sqrt{\pi}|\Gamma(-s)|
\Gamma(|m|+1+s)},\quad
m\in\Z\setminus\{0\}.
\end{equation}
\item \label{bb} If $u\in\ell_0$ then
$\displaystyle \lim_{s\to0^+}(-\Delta_h)^su_{j}=u_{j}.
$
If $u\in\ell^\infty$ then
$\displaystyle \lim_{s\to1^-}(-\Delta_h)^su_{j}=-\Delta_hu_{j}.$
\item \label{cc} For $0<s<1$ there exists a positive constant $C_{s}$ such that
\begin{equation}
\label{eq:frKernelEstimate2}
K_s(m) \leq \frac{C_{s}}{|m|^{1+2s}}.
\end{equation}
\end{enumerate}
\end{thm}

Observe that, in view of \eqref{eq:puntualdiscreta}, the fractional discrete Laplacian is a nonlocal operator on $\Z_h$. Notice also that our definition is neither
a direct discretization of the pointwise formula for the fractional Laplacian, nor
a ``discrete analogue'', but the $s$-fractional power of the discrete Laplacian.

\begin{rem}[On the negative powers]
We recall here the nice work by E.~M. Stein and S.~Wainger,
see \cite{Stein-Wainger}, in which  they consider a fractional integral type operator on $\Z$ given by convolution with the kernel
$|m|^{-(1-2s)}$. It is interesting to notice that if we define, again with the help of the semigroup language,
the negative powers (or fractional integrals) $(-\Delta_h)^{-s}$ then the kernel $K_{-s}$ of this operator
satisfies $K_{-s}(m) \le C_s|m|^{-(1-2s)}$. Moreover
in the case $0<s< 1/2,$ there exist positive constants $c_s$ and $C_s$
such that $|K_{-s}(m)-c_s|m|^{-(1-2s)}| \le C_s|m|^{-(2-2s)}$, see Section~\ref{sec:dicreteFractional}.
\end{rem}

\begin{rem}[Probabilistic interpretation]
Let $u$ be a discrete harmonic function on $\Z_h$, that is,
$-\Delta_hu=0$. This is equivalent to
\[
u_{j} = \frac12u_{j+1} +\frac12 u_{j-1}.
\]
This mean value identity says that a
discrete harmonic function describes the movement of a particle that jumps either to the adjacent left
point or to the adjacent right point with probability $1/2$.
Suppose now that $u_j$ is a fractional discrete harmonic function, that is, $(-\Delta_h)^su_j=0$. Then
we have the following mean value identity:
\[
u_j=\frac{1}{\Sigma_s}\sum_{m\neq j}u_{m}K_s(j-m),
\]
where $\Sigma_s=\sum_{m\neq 0}K_s(m)=\frac{2^{2s}\Gamma(1/2+s)}{\sqrt{\pi}\,\Gamma(1+s)}$.
In a parallel way we understand this last identity as saying that a
fractional discrete harmonic function describes a particle
that is allowed to jump to any point on $\Z_h$ (not only to the adjacent ones)
and that the probability to jump from the point $j$ to the point $m$ is $\frac{1}{\Sigma_s}K_s(j-m)$.
As $s\to1^-$ the probability to jump from $j$ to a non adjacent point tends to zero, while the probability
to jump to an adjacent point tends to one, recovering in this way the previous situation. As $s\to0^+$,
the probability to jump to any point tends to zero, so there are no jumps.
\end{rem}

Formula \eqref{eq:puntualdiscreta} makes it possible to prove H\"older
estimates parallel to the corresponding estimates proved for the classical Laplacian
in~\cite{Silvestre-CPAM}. The result we get is the following.
For the definition of discrete H\"older spaces $C^{k,\beta}_h$
see Section~\ref{sec:ProofConsistency}.

\begin{thm}
\label{thm:Holder}
Let $0<\beta\le1$ and $0<s<1$.
\begin{enumerate}[(i)]
\item \label{(i)} Let $u\in C^{0,\beta}_h$ and $2s<\beta$. Then
$(-\Delta_h)^su\in C^{0,\beta-2s}_h$ and
$$\|(-\Delta_h)^su\|_{C^{0,\beta-2s}_h}\leq C\|u\|_{C^{0,\beta}_h}.$$
\item \label{(ii)} Let $u\in C^{1,\beta}_h$ and $2s<\beta$. Then
$(-\Delta_h)^su\in C^{1,\beta-2s}_h$ and
$$\|(-\Delta_h)^su\|_{C^{1,\beta-2s}_h}\leq C\|u\|_{C^{1,\beta}_h}.$$
\item \label{(iii)} Let $u\in C^{1,\beta}_h$ and $2s>\beta$. Then
$(-\Delta_h)^su\in C^{0,\beta-2s+1}_h$ and
$$\|(-\Delta_h)^su\|_{C^{0,\beta-2s+1}_h}\leq C\|u\|_{C^{1,\beta}_h}.$$
\item \label{(iv)} Let $u\in C^{k,\beta}_h$ and assume that $k+\beta-2s$ is not an integer, with $2s<k+\beta$. Then
$(-\Delta_h)^su\in C^{l,\gamma}_h$, where $l$ is the integer part of $k+\beta-2s$ and $\gamma=k+\beta-2s-l$.
\end{enumerate}
The constants $C>0$ above are independent of $h$ and~$u$.
\end{thm}

Although the proof of Theorem~\ref{thm:consistencia1d}\eqref{(iCon)} is not trivial,
one could say in a na\"ive way that such an approximation result
is in some sense announced by the results in Theorem~\ref{thm:Holder}. Indeed, 
the fractional discrete Laplacian maps $C^\alpha_h$ into $C^{\alpha-2s}_h$. The continuous version
of this property is also true for the fractional Laplacian, so the restriction of $(-\Delta)^su$
to the mesh $\Z_h$ is in $C^{\alpha-2s}_h$
whenever $u\in C^\alpha(\R)$.

So far the results shown above were developed in the one dimensional case. Analogous results to Theorem~\ref{thm:basicProperties} and Theorem~\ref{thm:Holder} can be extended to higher dimensions. These multidimensional results will appear elsewhere. Our techniques depend strongly on the explicit expression for the kernel of the fractional operators. Nevertheless, we cannot find such explicit expressions for the kernels in the two
(or higher) dimensional case. This yields technical difficulties that we cannot overcome to prove a multidimensional version of Theorem~\ref{thm:consistencia1d}. However, we have been able to get asymptotics for such kernels
in dimension two that we believe could also be reproduced for the multidimensional case.

At the end of the paper we will show some pictures of examples drawn with \textit{Mathematica} that illustrate
our fractional discrete operators, both in one and two dimensions. We take examples of functions $u$
in the continuous variable for which $f:=(-\Delta)^su$
is explicitly known. We plot $f$ together with the discrete function $f_j=(-\Delta_h)^s(r_hu)_j$.
We also consider some explicit solutions $u=(-\Delta)^{-s}f$ for given~$f$.
We compare the graph of $u$ with the solution $u_j=(-\Delta_h)^{-s}(r_hf)_j$.
In this regard, we are addressing the question of whether the Poisson problem for the fractional Laplacian
$$(-\Delta)^su=f,\quad\hbox{in}~\R^n,$$
can be discretized by using our formulas, and whether the solutions to this discretized problem converge in some sense to the solutions of the continuous Poisson problem.

Several authors have been interested in solving the fractional Poisson equation
in a numerical or discrete way.
We mention here the recent works \cite{Bonito-Pasciak, DelTeso, Huang-Oberman, Nochetto},
see also the references therein, for the numerical approach, and
\cite{Zoia} for a one-dimensional discrete approach.

As we remarked earlier, the strategy used to obtain our results is the language of semigroups. Since the
semidiscrete heat semigroup is given in terms of modified Bessel functions, see Section~\ref{sec:proofBasic},
we will use exhaustively some properties and facts about these functions.

The structure of the paper is as follows.
Section~\ref{sec:proofBasic} is devoted to the proofs of Theorems~\ref{thm:basicProperties} and~\ref{thm:Holder}.
The proof of Theorem~\ref{thm:consistencia1d} is included in Section~\ref{sec:ProofConsistency}.
In Section~\ref{sec:dicreteFractional} we analyze the negative powers $(-\Delta_h)^{-s}$. The asymptotics for the kernel of the fractional discrete Laplacian and fractional integral in dimension two are studied in Section~\ref{sec:Mellin}.
In Sections~\ref{sec:oneDim} and~\ref{sec:twoDim} we show our pictures. Finally, some technical proofs
needed on the way and the properties of the Bessel functions are collected in Section~\ref{sec:technical}.
By $C_s$ we mean a positive constant depending on $s$ that may change in each occurrence, while
by $C$ we will denote a constant independent of the significant variables.

\section{The fractional discrete Laplacian: proofs of Theorem~\ref{thm:basicProperties} and Theorem~\ref{thm:Holder}}
\label{sec:proofBasic}

Given $u_j:\Z_h\to\R$, the solution to the semidiscrete heat equation with
initial discrete temperature $u_j$ can be written as
\begin{equation}
\label{semigrupo}
e^{t\Delta_h}u_{j} = \sum_{m\in \Z}G\big(j-m,\tfrac{t}{h^2}\big)u_{m}=
\sum_{m\in \Z}G\big(m,\tfrac{t}{h^2}\big)u_{j-m}, \quad t\in [0,\infty),
\end{equation}
where the semidiscrete heat kernel $G$ is defined as
$$G(m,t) = e^{-2t} I_{m}(2t),\quad m\in \Z.$$
This follows by scaling from the case $h=1$ of \cite{FiveGuys,GrIl}.
Here, $I_\nu$ is the modified Bessel function of order $\nu$ whose properties are collected in Section \ref{sec:technical}.
By \eqref{eq:negPos} and \eqref{eq:Ik>0} the kernel $G(m,t)$ is
symmetric in $m$, that is, $G(m,t)=G(-m,t)$, and positive.

\begin{proof}[Proof of Theorem~\ref{thm:basicProperties}]

First we will check that $e^{t\Delta_h}u_{j}$ is well defined. Indeed, if $N>0$, by using the asymptotic of the Bessel function for large order~\eqref{eq:asymptotics-order-large}, then
\begin{align*}
\sum_{|m|>N}G\big(m,\tfrac{t}{h^2}\big)|u_{j-m}| &\le
Ce^{-2t} \sum_{|m-j|>N} \frac{t^{|m|}(1+|m-j|)^{1+2s}}
{\Gamma(|m|+1)} \frac{|u_{m-j}|}{(1+|m-j|)^{1+2s}}
\leq C_{t,s,N,j}\|u\|_{\ell_s}.
\end{align*}

\noindent\textit{(a).} Define
\begin{equation} \label{eq:kernelFractionalLaplacian}
K_s(m) = \frac1{|\Gamma(-s)|} \int_0^\infty G(m,t) \frac{dt}{t^{1+s}}
\end{equation}
for $m\neq 0$, and $K_s(0) =0$. The symmetry of this kernel in $m$ follows from the symmetry of $G(m,t)$.
Therefore it is enough to assume that $m\in \N$. To get~\eqref{eq:fractionalKernelOned}, we use \eqref{eq:rusos} with $c=2$ and $\nu=m$. On the other hand, it is easy to show (see for example \cite{FiveGuys} for the case $h=1$) that
 $e^{t\Delta_h}1\equiv1$. Hence, from \eqref{definition} and~\eqref{semigrupo},
\begin{align*}
(-\Delta_h)^su_{j}
&= \frac{1}{\Gamma(-s)} \int_0^\infty \sum_{m\neq j}
G\big(j-m,\tfrac{t}{h^2}\big)(u_{m}-u_{j}) \frac{dt}{t^{1+s}} \\
&= \frac{1}{\Gamma(-s)} \sum_{m\neq j} (u_{m}-u_{j})\int_0^\infty
G\big(j-m,\tfrac{t}{h^2}\big) \frac{dt}{t^{1+s}} \\
&= \frac{1}{h^{2s}|\Gamma(-s)|} \sum_{m\neq j}(u_{j}-u_{m})
\int_0^\infty G(j-m,r) \frac{dr}{r^{1+s}} \\
&= \frac{1}{h^{2s}} \sum_{m\neq j}(u_{j}-u_{m})K_{s}(j-m).
\end{align*}
For the interchange of summation and integration in the second equality, we consider the two terms
\[
\int_0^\infty \sum_{m\neq j}
G\big(j-m,\tfrac{t}{h^2}\big)|u_{m}| \frac{dt}{t^{1+s}}
+|u_{j}| \int_0^\infty \sum_{m\neq j}
G\big(j-m,\tfrac{t}{h^2}\big)\frac{dt}{t^{1+s}}.
\]
By using \eqref{eq:frKernelEstimate2} we see that
the first term above is bounded
by $C_{s}\sum_{m\neq j}|m-j|^{-(1+2s)}|u_{m}|$, which is finite for each $j$
because $u\in \ell_{s}$. For the second term we use again~\eqref{eq:frKernelEstimate2}.

\noindent\textit{(b).} Suppose first that $u\in\ell_0$. We have
\[
h^{2s}(-\Delta_h)^su_{j} =
u_{j} \sum_{m\neq j} K_s(j-m) - \sum_{m\neq j}K_s(j-m)u_{m} =: u_{j}T_1-T_2.
\]
We write $T_1=T_{1,1}+T_{1,2}$, where
\[
T_{1,2} = \frac{1}{|\Gamma(-s)|} \sum_{m\neq j} \int_1^\infty G(j-m,t)\,\frac{dt}{t^{1+s}}=
\frac{1}{|\Gamma(-s)|} \sum_{m\neq 0} \int_1^\infty G(m,t)\,\frac{dt}{t^{1+s}}.
\]
We are going to prove that $T_{1,1}$ and $T_2$ tend to zero, while $T_{1,2}$ tends to 1,
as $s\to0^+$. Let us begin with $T_{1,2}$. By adding and subtracting the term $m=0$ in the sum
and using~\eqref{eq:sumIk}, we get
\[
T_{1,2} = \frac{1}{|\Gamma(-s)|}\left(\frac{1}{s}-\int_1^\infty \frac{e^{-2t} {I_0(2t)}}{t^{1+s}}\,dt\right).
\]
By noticing that $|\Gamma(-s)|s=\Gamma(1-s)$ and that, by~\eqref{eq:asymptotics-infinite}, we have
\[
\frac{1}{|\Gamma(-s)|} \int_1^\infty \frac{e^{-2t} {I_0(2t)}}{t^{1+s}}\,dt
\le \frac{C}{|\Gamma(-s)|} \int_{1}^{\infty}t^{-1/2-1-s}\,dt
= \frac{C}{|\Gamma(-s)|(1/2+s)},
\]
we get $T_{1,2}\to 1$ as $s\to0^+$, as desired. Next we handle the other two terms $T_{1,1}$ and $T_2$.
On one hand, by~\eqref{eq:asymptotics-zero-ctes},
\begin{align*}
T_{1,1} &\sim \frac{1}{|\Gamma(-s)|} \sum_{m\neq 0}
\frac{1}{\Gamma(|m|+1)} \int_0^1e^{-2t}t^{|m|}
\,\frac{dt}{t^{1+s}}
\le \frac{1}{|\Gamma(-s)|}
\sum_{m\neq 0} \frac{1}{\Gamma(|m|+1)} \frac{1}{|m|-s},
\end{align*}
and the last quantity tends to $0$ as $s\to 0^+$. On the other hand,
for $T_2$, we use~\eqref{eq:fractionalKernelOned} to obtain
\[
|T_2|\le \frac{C_{s}}{|\Gamma(-s)|} \sum_{m\neq j}
\frac{\Gamma(|j-m|-s)}{\Gamma(|j-m|+1+s)}|u_{m}|.
\]
The constant $C_{s}$ remains bounded as $s\to0^+$.
Since $u\in \ell_0$, by dominated convergence, we have that the sum
above is bounded by $\|u\|_{\ell_0}$, for each $j\in \Z$. Therefore $T_2$ tends to $0$ as $s\to 0^+$.

For the second limit, suppose that $u\in\ell^\infty$. By using the symmetry of the kernel $K_s$ we can write
\[
h^{2s}(-\Delta_h)^s u_{j} = S_1+S_2,
\]
where
\[
S_1 = K_s(1) \bigl(-u_{j+1}+2u_{j}-u_{j-1}\bigr),\quad\text{and}\quad
S_2 = \sum_{|m|>1}K_s(m)(u_j-u_{j-m}).
\]
Next we show that $K_s(1)\to1$, while $S_2\to0$, as $s\to1^-$, which would give the conclusion. By~\eqref{eq:fractionalKernelOned} we have
\[
\lim_{s\to 1^-}K_s(1) = \lim_{s\to 1^-} \frac{4^s\Gamma(1-s)\Gamma(1/2+s)}{\sqrt{\pi}|\Gamma(-s)|\Gamma(2+s)}
= \lim_{s\to 1^-}\frac{4^s\Gamma(1/2+s)}{\sqrt{\pi}\Gamma(2+s)}
= \frac{4\Gamma(3/2)}{\sqrt{\pi}\Gamma(3)}=1.
\]
On the other hand, by~\eqref{eq:fractionalKernelOned}, $S_2$ is bounded by
\begin{align*}
2\|u\|_{\ell^\infty}\sum_{|m|>1}K_s(m)\le 2\|u\|_{\ell^\infty}
\frac{4^s\Gamma(\frac{1}{2}+s)}{\pi^{1/2}|\Gamma(-s)|}
\sum_{|m|>1}^{\infty}
\frac{\Gamma(|m|-s)}{\Gamma(|m|+1+s)},
\end{align*}
which goes to zero as $s\to1^-$.

\noindent\textit{(c).} Observe that the estimate in \eqref{eq:frKernelEstimate2} follows from \eqref{eq:fractionalKernelOned} and
Lemma~\ref{eq:lowerBound}.
\end{proof}

Consider the first order difference operators on $\Z_h$ given as (observe that we already introduced some definitions in the statement of Theorem \ref{thm:consistencia1d}).
\[
D_+u_{j}:=\frac{u_{j+1}-u_{j}}{h},\quad\text{and}\quad
D_-u_{j}:=\frac{u_{j}-u_{j-1}}{h}.
\]
For $\gamma, \eta\in \N_0$,
we define $D_{+,-}^{\gamma,\eta}u_j:=D_{+}^{\gamma}
D_{-}^{\eta}u_j$. Here, by $D_{\pm}^{k}$ we mean that we apply $k$ times the operator~$D_{\pm}$
and $D_{\pm}^0$ is the identity operator.

\begin{defn}[Discrete H\"older spaces]
Let $0<\beta\le1$ and $k\in \N_0$. A bounded
function $u:\Z_h\to\R$ belongs to the discrete H\"older space $C_h^{k,\beta}$ if
\[
[D_{+,-}^{\gamma,\eta}u]_{C_h^{0,\beta}} :=
\sup_{m\neq j} \frac{|D_{+,-}^{\gamma,\eta}u_{j}-
D_{+,-}^{\gamma,\eta} u_{m}|}{h^{\beta}|j-m|^{\beta}}\le C<\infty
\]
for each pair $\gamma,\eta\in \N_0$ such that $\gamma+\eta=k$.
The norm in the spaces $C_h^{k,\beta}$ is defined in the usual way:
\[
\|u\|_{C_h^{k,\beta}} := \max_{\gamma+\eta\le k}\sup_{m\in \Z_h}|D_{+,-}^{\gamma,\eta}u_m|
+\max_{\gamma+\eta=k}[D_{+,-}^{\gamma,\eta}u]_{C_h^{0,\beta}}.
\]
\end{defn}

\begin{proof}[Proof of Theorem~\ref{thm:Holder}]
By iteration of \textit{(i)}, \textit{(ii)} and \textit{(iii)}
we can get~\textit{(iv)}. Let us start with the proof of~\textit{(i)}.
Let $k,j\in \Z$. We have
\begin{equation}
\label{diferenciaS}
|(-\Delta_h)^su_{k}-(-\Delta_h)^su_{j}|= \frac{1}{h^{2s}}|S_1+S_2|,
\end{equation}
where
\begin{equation}
\label{S1}
S_1 := \sum_{1\leq|m|\le|k-j|}\big(u_{k}-u_{k+m}-u_{j}+u_{j+m}\big)K_s(m),
\end{equation}
and $S_2$ is the rest of the sum over $|m|>|k-j|$.
By~\eqref{eq:frKernelEstimate2},
\[
S_1 \le C_{s}[u]_{C_h^{0,\beta}}h^{\beta} \sum_{1\leq|m|\le|k-j|}
\frac{|m|^{\beta}}{|m|^{1+2s}} \le C_{s}
[u]_{C_h^{0,\beta}}h^{\beta}|k-j|^{\beta-2s}.
\]
For $S_2$ we use that $|u_{k}-u_{j}|\le [u]_{C_h^{0,\beta}}h^{\beta}|k-j|^{\beta}$ and
\eqref{eq:frKernelEstimate2} again to get
\[
S_2\le C_{s}[u]_{C_h^{\beta}}h^{\beta}|k-j|^{\beta} \sum_{|m|>|k-j|}|m|^{-1-2s}\le C_{s}
[u]_{C_h^{0,\beta}}h^{\beta}|k-j|_2^{\beta-2s}.
\]

For \textit{(ii)}, it suffices to prove that $D_{\pm }(-\Delta_h)^su\in C^{0,\beta-2s}_h$.
For this it is enough to observe that $D_{\pm }$
commutes with $(-\Delta_h)^s$ and then use \textit{(i)}.

For \textit{(iii)} we are going to use~\eqref{diferenciaS}. Without loss of generality, take $m\in \N$. We split the sum in \eqref{diferenciaS}
by taking the terms $u_{k}-u_{k+m}$ and $u_j-u_{j+m}$ separately.
The following computation works for both terms, so we do it only for the first one. It is verified that
\begin{equation}
\label{eq:discreteTVM}
u_{k+m}-u_{k}=h\sum_{\gamma=0}^{m-1}D_+u_{k+\gamma}.
\end{equation}
Therefore,
\begin{equation}
\label{vaso}
u_{k}-u_{k+m}=\Big(hmD_{+}u_{k}-
h\sum_{\gamma=0}^{m-1}D_+u_{k+\gamma}\Big)
-hmD_{+}u_{k}.
\end{equation}
On one hand, by taking into account that the kernel $K_s(m)$ is even, we get
\begin{equation}
\label{eq:nulo}
\sum_{1\leq|m|\le|k-j|}(hm D_{+}u_{k})K_s(m)
=hD_{+}u_{k}\sum_{1\leq|m|\le|k-j|}mK_s(m)=0.
\end{equation}
On the other hand, since $u\in C_h^{1,\beta}$,
the first term in the right hand side of \eqref{vaso} is bounded by
\begin{equation}
\label{eq:truco}
h\sum_{\gamma=0}^{m-1}\big|D_{+}u_{k}-D_{+}u_{k+\gamma}\big|\le h^{1+\beta}[u]_{C^{1,\beta}_h}
\sum_{\gamma=0}^{m-1}|\gamma|^{\beta}\le h^{1+\beta}[u]_{C^{1,\beta}_h}|m|^{\beta}|m|=[u]_{C^{1,\beta}_h}(h|m|)^{1+\beta}.
\end{equation}
Pasting \eqref{eq:nulo} and~\eqref{eq:truco} (and their analogous for $u_{j}-u_{j+m}$)
into~\eqref{S1}, we conclude that
\[
|S_1|\le C_{s}[u]_{C^{1,\beta}_h}h^{1+\beta}\sum_{1\leq|m|\le|k-j|}
\frac{|m|^{1+\beta}}{|m|^{1+2s}}
\le C_{s}[u]_{C^{1,\beta}_h}(h|k-j|)^{1+\beta-2s}.
\]
Now we deal with $S_2$. By~\eqref{eq:discreteTVM},
\begin{align*}
\big|(u_{k}-u_{j})-(u_{k+m}-u_{j+m})\big|&=\big|(u_{(k-j)+j}-u_{j})-(u_{(k-j)+(j+m)}-u_{j+m})\big|\\
&\le h\sum_{\gamma=0}^{k-j-1}|D_{+}u_{j+\gamma}-D_{+}u_{j+m+\gamma}|\le [u]_{C^{1,\beta}_h}h^{1+\beta}|m
|^{\beta}|k-j|.
\end{align*}
Hence,
\[
|S_2| \le C[u]_{C^{1,\beta}_h}h^{1+\beta}|k-j| \sum_{|m|>|k-j|}|m|^{\beta}K_s(m)\le C[u]_{C^{1,\beta}_h}h^{1+\beta}|k-j|^{1+\beta-2s}.
\qedhere
\]
\end{proof}

\section{Error of approximation: Proof of Theorem~\ref{thm:consistencia1d}}
\label{sec:ProofConsistency}

Recall that a continuous real function $u$ belongs to the H\"older space $C^{k,\beta}(\R)$ if $u\in C^k(\R)$ and
\[
[u^{(k)}]_{C^{0,\beta}(\R)} := \sup_{\substack{x,y\in \R\\x\neq y}} \frac{|u^{(k)}(x)-u^{(k)}(y)|}{|x-y|^{\beta}}<\infty,
\]
where we are using $u^{(k)}$ to denote the $k$-th derivative of~$u$. The norm in the spaces $C^{k,\beta}(\R)$ is
\[
\|u\|_{C^{k,\beta}(\R)} := \sum_{l=0}^k\|u^{(l)}\|_{L^\infty(\R)}+[u^{(k)}]_{C^{0,\beta}(\R)}.
\]
In order to prove Theorem~\ref{thm:consistencia1d} we need a preliminary lemma. Let us define
the following constant:
\begin{equation}
\label{eq:constante1d}
c_{s} := \frac{4^s\Gamma(1/2+s)}{\pi^{1/2}|\Gamma(-s)|}>0.
\end{equation}

\begin{lem}
\label{lem:compa}
Let $0<s<1$. Given $j\in \Z$, we have
\begin{equation}
\label{eq:lem1}
\bigg|c_s\int_{|y-h(j+m)|<h/2}\frac{dy}{|hj-y|^{1+2s}}-\frac{K_s(m)}{h^{2s}}\bigg|\le \frac{C_s}{h^{2s}m^{2+2s}}, \quad \text{for all }\, m\in \Z,
\end{equation}
\begin{equation}
\label{eq:lem2}
\int_{|y-h(j+m)|<h/2}\frac{dy}{|hj-y|^{1+2s}}\le \frac{C_s}{h^{2s}m^{1+2s}},\quad \text{for all }\, m\in \Z,
\end{equation}
and
\begin{equation}
\label{eq:lem3}
\sum_{m\in \Z}\int_{|y-h(j+m)|<h/2}\frac{hj-y}{|hj-y|^{1+2s}}\,dy = 0.
\end{equation}
\end{lem}

\begin{proof}
The change of variable $hj-y=hz$ produces
\begin{multline*}
\bigg|\frac{c_s}{h^{2s}}\int_{|z-m|<1/2}\frac{dz}{|z|^{1+2s}}-\frac{K_s(m)}{h^{2s}}\bigg| \\
\le \bigg|\frac{c_s}{h^{2s}}\int_{|z-m|<1/2} \bigg(\frac{1}{|z|^{1+2s}}-\frac{1}{|m|^{1+2s}}\bigg)\,dz\bigg|
+ h^{-2s}\bigg|\frac{c_s}{|m|^{1+2s}}-K_s(m)\bigg|.
\end{multline*}
By using the Mean Value Theorem,
\[
\bigg|\int_{|z-m|<1/2} \bigg(\frac{1}{|z|^{1+2s}}-\frac{1}{|m|^{1+2s}}\bigg)\,dz\bigg|
\le C_s\bigg|\int_{|z-m|<1/2}\frac{dz}{|m|^{2+2s}}\bigg| = \frac{C_s}{|m|^{2+2s}},
\]
while by Lemma~\ref{eq:lowerBound},
\[
\bigg|\frac{c_s}{|m|^{1+2s}}-K_s(m)\bigg|\le \frac{C_s}{|m|^{2+2s}}.
\]
Thus \eqref{eq:lem1} follows. For \eqref{eq:lem2}, it is easy to see that
\[
\int_{|y-(h(j+m))|<h/2}\frac{dy}{|hj-y|^{1+2s}}
\le C_s\int_{|y-(h(j+m))|<h/2}\frac{dy}{|hm|^{1+2s}} = \frac{C_s}{h^{2s}|m|^{1+2s}}.
\]
Finally, let us prove~\eqref{eq:lem3}. By symmetry, we have
\[
\int_{|y-hj|<h/2}\frac{(hj-y)}{|hj-y|^{1+2s}}\,dy = 0.
\]
Moreover, by changing variables $hj-y=z$, we get
\begin{align*}
\sum_{\substack{m\in \Z\\ m\neq 0}} 
\int_{|z-hm|< h/2} \frac{z}{|z|^{1+2s}}\,dz
&= \sum_{\substack{\ell\in \Z\\ \ell\neq 0}}
\int_{|z+h\ell|< h/2} \frac{z}{|z|^{1+2s}}\,dz
= \sum_{\substack{\ell\in \Z\\ \ell\neq 0}}
\int_{|u-h\ell|< h/2} \frac{-u}{|u|^{1+2s}}\,du,
\end{align*}
and the conclusion readily follows.
\end{proof}

\begin{proof}[Proof of Theorem~\ref{thm:consistencia1d}]
\textit{(i)}. We write, for each $j\in \Z$,
\begin{align*}
r_h\big((-\Delta)^su\big)_j
&= c_{s}\sum_{m\in \Z}\int_{|y-h(j+m)|< h/2} \frac{u(hj)-u(y)}{|hj-y|^{1+2s}}\,dy\\
&= c_{s}\bigg(\sum_{m\in \Z}\int_{|y-h(j+m)|< h/2} \frac{u(h(j+m))-u(y)}{|hj-y|^{1+2s}}\,dy\\
&\qquad\quad + \sum_{\substack{m\in \Z\\ m\neq 0}} \big(u(hj)-u(h(j+m))\big)
\int_{|y-h(j+m)|< h/2} \frac{dy}{|hj-y|^{1+2s}}\bigg)\\
& =: c_s(S_1+S_2).
\end{align*}
By using that $u\in C^{0,\beta}(\R)$ and~\eqref{eq:lem2}, we have
\begin{align*}
|S_1| &\le C[u]_{C^{0,\beta}(\R)} \sum_{m\in \Z} \int_{|y-h(j+m)|< h/2} \frac{h^{\beta}\,dy}{|hj-y|^{1+2s}} \\
&\le C_s [u]_{C^{0,\beta}(\R)}h^{\beta}\sum_{m\in \Z} \frac{1}{h^{2s}|m|^{1+2s}}
= C_s[u]_{C^{0,\beta}(\R)}h^{\beta-2s}.
\end{align*}
Now we compare $S_2$ with $(-\Delta_h)^s(r_hu)_j$.
Since $u\in C^{0,\beta}(\R)$, by Lemma~\ref{lem:compa} we can see that
\begin{align*}
\bigg|c_s\sum_{\substack{m\in \Z\\ m\neq 0}}
\big(u(hj)&-u(h(j+m))\big) \int_{|y-h(j+m)|< h/2} 
\frac{dy}{|hj-y|^{1+2s}}-(-\Delta_h)^s(r_hu)_j\bigg|\\
&\le \sum_{\substack{m\in \Z\\m\neq 0}}
\big|u(hj)-u(h(j+m))\big| \bigg|c_s\int_{|y-h(j+m)|< h/2} \frac{dy}{|hj-y|^{1+2s}}-\frac{K_s(m)}{h^{2s}} \bigg| \\ 
&\le C_s[u]_{C^{0,\beta}(\R)}\sum_{\substack{m\in \Z\\m\neq 0}} \frac{|hm|^{\beta}}{h^{2s}|m|^{2+2s}}
\le C_s[u]_{C^{0,\beta}(\R)}h^{\beta-2s}.
\end{align*}

\textit{(ii)}. Observe that
$\frac{d}{dx}$ commutes with $(-\Delta)^s$ and $D_+$ with $(-\Delta_h)^s$. Then
\begin{multline*}
\big\|D_+(-\Delta_h)^s (r_h u)- r_h\big(\tfrac{d}{dx}(-\Delta)^s u\big) \big\|_{\ell^\infty} \\
\le  \big\|(-\Delta_h)^s D_+(r_h u) - (-\Delta_h)^s \big(r_h \tfrac{d}{dx}u\big) \big\|_{\ell^\infty}
+ \big\| (-\Delta_h)^s \big(r_h \tfrac{d}{dx}u\big) - r_h\big(\tfrac{d}{dx}(-\Delta)^s u\big) \big\|_{\ell^\infty}.
\end{multline*}
For the second term we just apply \textit{(i)}. As for the first one, by using the Mean Value Theorem,
\begin{align*}
\big| (-\Delta_h)^s D_+(r_h u)_j- &(-\Delta_h)^s  \big(r_h \tfrac{d}{dx}u \big)_j \big| \\
&= \bigg|\frac1{h^{2s}}\sum_{m\neq 0} K_s(m) \bigg[ \bigg( \frac{u(h(j+1))- u(hj)}{h} - u'(hj) \bigg) \\
&\qquad\qquad-\bigg(\frac{u(h(j+m+1))- u(h(j+m)))}{h}- u'(h(j+m))\bigg)\bigg]\bigg| \\
&= \bigg|\frac1{h^{2s}}\sum_{m\neq 0} K_s(m) \bigg[\big(u'(\xi_j)-u'(hj)\big) - 
\big( u'(\xi_{j+m})- u'(h(j+m)) \big)\bigg]\bigg| \\
&\le C\frac1{h^{2s}}\sum_{m\neq 0} K_s(m)  h^{\beta} \le C h^{\beta-2s},
\end{align*}
where $\xi_j$ is an intermediate point between $hj$ and $h(j+1)$ and analogously $\xi_{j+m}$. In the last inequality we have used the  hypothesis on the regularity of $u$.

\textit{(iii)}. By taking into account~\eqref{eq:lem3}, we can write
\begin{align*}
& r_h\big((-\Delta)^su\big)_j
= c_{s}\sum_{m\in \Z}\int_{|y-h(j+m)|< h/2} \frac{u(hj)-u(y)-u'(hj)(hj-y)}{|hj-y|^{1+2s}}\,dy \\
& \quad= c_{s}\bigg(\sum_{m\in \Z} \int_{|y-h(j+m)|< h/2} \frac{u(h(j+m))-u(y)-u'(hj)(h(j+m)-y)}{|hj-y|^{1+2s}}\,dy \\
& \qquad\quad+ \sum_{\substack{m\in \Z\\ m\neq 0}}
\big(u(hj)-u(h(j+m))-u'(hj)(hj-h(j+m))\big) \int_{|y-h(j+m)|< h/2} \frac{dy}{|hj-y|^{1+2s}}\bigg) \\
& \quad=: c_s(S_1+S_2).
\end{align*}
By the hypotheses and~\eqref{eq:lem2}, we have
\[
|S_1|\le C_s[u]_{C^{1,\beta}(\R)}\sum_{m\in \Z}|hm|^{\beta}h \frac{1}{h^{2s}|m|^{1+2s}} = C_s[u]_{C^{1,\beta}(\R)}h^{1+\beta-2s}.
\]
We compare $c_sS_2$ with $(-\Delta_h)^s(r_hu)_j$. Since $K_s(m)$ is even in $m$, we can write
\[
(-\Delta_h)^s(r_hu)_j=\sum_{\substack{m\in \Z\\ m\neq 0}} \big(u(hj)-u(h(j+m))-u'(hj)(hj-h(j+m))\big)\frac{K_s(m)}{h^{2s}}.
\]
Then \eqref{eq:lem1} and the hypothesis on $u$ give the result.

\textit{(iv)}. The proof in this case follows as in \textit{(ii)} by iteration $l$ times.
\end{proof}

\section{The negative powers or fractional discrete integral}
\label{sec:dicreteFractional}

Analogously to the fractional discrete Laplacian, the negative powers of the discrete Laplacian, also called the \textit{fractional discrete integral}, is defined for a discrete function $f:\Z_h\to\R$ and $j\in \Z$ as
$$(-\Delta_h)^{-s}f_{j}=\frac{1}{\Gamma(s)} \int_0^{\infty}e^{t\Delta_h}f_{j} \frac{dt}{t^{1-s}},\quad s>0.$$
By writing down the semidiscrete heat kernel and using Fubini's Theorem we readily see that
this is a convolution operator on $\Z$:
$$(-\Delta_h)^{-s}f_j=h^{2s} \sum_{m\in\Z}K_{-s}(j-m)f_{m},$$
where the discrete kernel $K_{-s}$ is given by 
$$K_{-s}(m)=\frac{1}{\Gamma(s)}  \int_0^{\infty}G(m,t) \frac{dt}{t^{1-s}}.$$
It is worth to compare this formula for $K_{-s}$ with the one for $K_s$, see \eqref{eq:kernelFractionalLaplacian}.

\begin{thm}
\label{lem:kernelFractionalIntegral}
Let $0<s<1/2$. We have the explicit expression
\begin{equation}
\label{eq:kernelFrIntegralOned}
K_{-s}(m) = \frac{4^{-s}\Gamma(1/2-s)\Gamma(|m|+s)}{\sqrt{\pi}\,\Gamma(s)
\Gamma(|m|+1-s)},~\hbox{for}~m\in \Z\setminus\{0\}, \quad K_{-s}(0) 
= \frac{4^{-s}\Gamma(1/2-s)}{\sqrt{\pi}\,\Gamma(1-s)}.
\end{equation}
Moreover, there exists a positive constant $C_s$ such that, for $m\in\Z\setminus\{0\}$,
\begin{equation}
\label{eq:growthFractional}
\Big| K_{-s}(m) - \frac{c_{s}}{|m|^{1-2s}}\Big| \le \frac{C_{s}}{|m|^{2-2s}}.
\end{equation}
\end{thm}

\begin{proof}
One can compute the integral defining the kernel explicitly, for $0<s<1/2$, just by using formula~\eqref{eq:rusos}. The estimate in~\eqref{eq:growthFractional} follows from \eqref{eq:kernelFrIntegralOned} and Lemma~\ref{eq:lowerBound}.
\end{proof}

As we mentioned in the Introduction,
a discrete \textit{analogue} of the usual continuous fractional integration can be defined for $f:\Z^n\to\R$
and $0<\lambda<n$ by
\[
I_{\lambda}f_j = \sum_{\begin{smallmatrix}m\in\Z^n\\m\neq0\end{smallmatrix}}\frac{1}
{|m-j|^\lambda}f_{m}, \quad j\in\Z^n.
\]
E. M. Stein and S. Wainger obtained $\ell^p\to \ell^q$ estimates for this operator in~\cite{Stein-Wainger}. In view of our expression for the kernel $K_{-s}$ in \eqref{eq:kernelFrIntegralOned} and \eqref{eq:lowerFrac} of Lemma~\ref{eq:lowerBound}, we notice that the one-dimensional version of the operator $I_{1-2s}$
controls, up to a factor of $h$, our fractional discrete integral $(-\Delta_h)^{-s}$. Therefore, mapping properties for $(-\Delta_h)^{-s}$ can be deduced from the results in~\cite{Stein-Wainger}.

\section{Asymptotics for the two dimensional fractional discrete Laplacian}
\label{sec:Mellin}

In this section, we present an expression for the
two dimensional kernel $K_s(m)$, $m=(m_1,m_2)\in\Z^2$, when both $|m_i|\to \infty$.
We denote $\|m\|_2:=(|m_1|^2+|m_2|^2)^{1/2}$, for $m\in \Z^2$.

\begin{lem}
\label{lem:expressionKs2d}
Let $0<s<1$. Let $K_s$ be the kernel in \eqref{eq:kernelFractionalLaplacian}
in dimension two. Then
\begin{equation}
\label{eq:frKernelEstimate2d}
0<K_s(m) = \frac{c_{2,s}}{\|m\|_2^{2+2s}}+\text{higher order terms},\quad
\text{when both}~m_i\to \pm\infty, \quad  i=1,2,
\end{equation}
where
\begin{equation*}
c_{2,s} = \frac{4^s\Gamma(1+s)}{\pi|\Gamma(-s)|}.
\end{equation*}
\end{lem}

\begin{proof}
The ideas in this subsection are inspired in some results
and techniques from \cite{Chinta-Jorgenson-Karlsson}. First we list the ingredients that will be used.

In \cite[Chapter 4, (5.05)]{Olver} we find the following expansion for the ratio of two Gamma functions, for $z\in \R$, $z\to \infty$,
\begin{equation}
\label{eq:term2}
\frac{\Gamma(z+a)}{\Gamma(z+b)}= z^{a-b}\bigg(1+\frac{(a-b)(a+b-1)}{2z}+\frac{1}{12}
\binom{a-b}{2}\big(3(a+b)^2-7a-5b+2\big)
\frac{1}{z^2}\bigg)+E(a,b,z),
\end{equation}
where $E(a,b,z)$ is an integral (convergent if $3+b-a>0$) whose absolute value is bounded by $z^{-(3+b-a)}$ times a constant depending on $a$ and $b$.

We will use formula \eqref{eq:rusos} with $c=2$, $\alpha=-s$ (where $0<s<1$), and $\nu=m_i$ where $m_i\in \N$, $m_i\neq 0$. Indeed, we have
\begin{equation}
\label{eq:rusos-nuestro-caso}
\int_0^{\infty} e^{-2t} I_{m_i}(2t) t^{-s-1}\,dt
= \frac{4^{s}}{\sqrt{\pi}} \frac{\Gamma(1/2+s)\Gamma(m_i-s)}{\Gamma(m_i+1+s)}.
\end{equation}

Let
\begin{equation}
\label{eq:funciones-fi}
f_{m_i}(t)=e^{-2t}I_{m_i}(2t).
\end{equation}
Consider the Mellin transform
\begin{equation}
\label{eq:Mellin}
\widetilde{f}_{m_i}(s) = \int_0^{\infty}f_{m_i}(t) \frac{dt}{t^{1+s}}
\end{equation}
which is absolutely convergent, by~\eqref{eq:rusos-nuestro-caso}, for $0<s<1$. By the Mellin inversion theorem we can recover $f_{m_i}(t)$ via the inverse Mellin transform, see for instance \cite{Titchmarsh}, that gives
\begin{equation}
\label{eq:inverse-Mellin}
f_{m_i}(t) = \frac{1}{2\pi i} \int_{c-i\infty}^{c+i\infty}\widetilde{f}_{m_i}(s)t^s\,ds,
\end{equation}
which is well defined for $c$ a real value between $0$ and $1$; the integral above is a line integral over a vertical line in the complex plane.

We recall the elementary formulas
\begin{equation}
\label{eq:def-gamma}
\Gamma(w)a^{-w} = \int_0^{\infty}e^{-at}t^{w} \,\frac{dt}{t}, \quad \Re w>0,
\end{equation}
and
\begin{equation}
\label{eq:def-e}
e^{-t} = \sum_{k=0}^{\infty}t^k\, \frac{(-1)^k}{k!}.
\end{equation}

Let $f_{m_i}(t)$ be the functions defined in~\eqref{eq:funciones-fi}. By using \eqref{eq:Mellin} and \eqref{eq:inverse-Mellin} we compute the Mellin transform of the product of these functions. Observe that the asymptotics for the modified Bessel function \eqref{eq:asymptotics-zero} and \eqref{eq:asymptotics-infinite} ensure that the Mellin transform $\widetilde{f_{m_1}f_{m_2}}(s)$ is well defined for $0<s<1$. We write
\begin{align*}
|\Gamma(-s)|K_s(m)=\widetilde{f_{m_1}f_{m_2}}(s) &= \int_0^{\infty} f_{m_1}(t)f_{m_2}(t) \frac{dt}{t^{1+s}} \\
&= \int_0^{\infty} f_{m_2}(t) \bigg[\frac{1}{2\pi i}
\int_{c_1-i\infty}^{c_1+i\infty}\widetilde{f}_{m_1}(x_1)t^{x_1}\,dx_1\bigg] \frac{dt}{t^{1+s}} \\
&= \frac{1}{2\pi i} \int_{c_1-i\infty}^{c_1+i\infty}\widetilde{f}_{m_1}(x_1)
\bigg[ \int_0^{\infty} f_{m_2}(t) \frac{dt}{t^{1+s-x_1}}\bigg]dx_1 \\
&= \frac{1}{2\pi i} \int_{c_1-i\infty}^{c_1+i\infty} \widetilde{f}_{m_1}(x_1) \widetilde{f}_{m_2}(s-x_1) \, dx_1.
\end{align*}
Note that the integrals that define $\widetilde{f}_{m_1}(x_1)$ and $ \widetilde{f}_{m_2}(s-x_1)$, see~\eqref{eq:Mellin}, are absolutely convergent for $-m_1<-\Re x_1<1/2$ and $-m_2<\Re x_1-s<1/2$, respectively, with $\Re x_1=c_1$. Thus, it is justified that we can write each $f_{m_i}$ above as the inverse Mellin transform of $\widetilde{f}_{m_i}$, with $c_i\in(0,1/2)$. With the computation above, and by~\eqref{eq:Mellin}, \eqref{eq:funciones-fi} and~\eqref{eq:rusos-nuestro-caso},
we have that
\begin{equation}
\label{eq:churro-de-gammas}
\widetilde{f_{m_1}f_{m_2}}(s) = \frac{4^s}{2\pi^2 i} \int_{c_1-i\infty}^{c_1+i\infty} \Gamma(1/2+x_1)\Gamma(1/2-x_1+s)\frac{\Gamma(m_1-x_1)}{\Gamma(m_1+x_1+1)}
 \frac{\Gamma(m_2+x_1-s)}{\Gamma(m_2-x_1+s+1)} \,dx_1.
\end{equation}
Observe that we have an analogous expression just by interchanging the roles of $m_1$ and $m_2$.

By~\eqref{eq:term2}, when $m_1,m_2\to\infty$, we can write
\[
\frac{\Gamma(m_1-x_1)}{\Gamma(m_1+x_1+1)}
\cdot \frac{\Gamma(m_2+x_1-s)}{\Gamma(m_2-x_1+s+1)} 
= \mathit{I}+\mathit{II}+\mathit{III}+\mathit{IV},
\]
where
\[
\mathit{I} := \frac{1}{m_1^{2x_1+1}m_2^{-2x_1+1+2s}},
\quad \mathit{II} := \frac{1}{m_1^{2x_1+1}} \,\mathcal{O}\bigg(\frac{1}{m_2^{-2x_1+2s+3}}\bigg),
\]
\[
\mathit{III} := \frac{1}{m_2^{-2x_1+2s+1}} \,\mathcal{O}\bigg(\frac{1}{m_1^{2x_1+3}}\bigg),
\quad \mathit{IV} := \mathcal{O}\bigg(\frac{1}{m_1^{2x_1+3}}\bigg) \,\mathcal{O}\bigg(\frac{1}{m_2^{-2x_1+2s+3}}\bigg).
\]
The term $\mathit{I}$ is the leading term. We compute the integral in $x_1$ containing this term. In order to do that, we move the contour of integration to infinity (in the negative direction of the $\Re x_1$-axis) and then we use the residue theorem (the integrand has poles at $ x_1=-j-1/2$ with residues equal to  $(-1)^j/j!$):
\begin{align*}
S &:= \frac{1}{2\pi i} \int_{c_1-i\infty}^{c_1+i\infty}
m_1^{-1-2x_1}
m_2^{-1-2s+2x_1}  \Gamma(1/2+x_1) \Gamma(1/2-x_1+s)\,dx_1 \\
&= m_1^{-1}m_2^{-1-2s} \sum_{j=0}^{\infty} \Big(\frac{m_1^2}{m_2^2}\Big)^{j+1/2}
\frac{(-1)^j}{j!}\Gamma(1+j+s)\\
& = m_1^{-2-2s}\sum_{j=0}^{\infty}
\Big(\frac{m_1^2}{m_2^2}\Big)^{j+1+s}
\frac{(-1)^j}{j!}\Gamma(1+j+s) \\
&= m_1^{-2-2s} \sum_{j=0}^{\infty}
\int_{0}^{\infty}
e^{-\frac{m_2^2}{m_1^2}t}t^{1+j+s}
\,\frac{(-1)^j}{j!} \frac{dt}{t},
\end{align*}
where, in the last equality, we used~\eqref{eq:def-gamma} with $w=1+j+s$ and $a=\frac{m_2^2}{m_1^2}$.
Observe that, by~\eqref{eq:def-e}, we have
\begin{align*}
S &= m_1^{-2-2s}
\int_{0}^{\infty}
e^{-\big(\frac{m_2^2}{m_1^2}+1\big)t}t^{s}\,dt=\Gamma(1+s)m_1^{-2-2s}
\bigg(\frac{m_2^2}{m_1^2}+1\bigg)^{-(1+s)}= \frac{\Gamma(1+s)}{(m_1^2+m_2^2)^{1+s}}\,j.
\end{align*}
We obtain that the integral coming from~\eqref{eq:churro-de-gammas} involving $\mathit{I}$ is $\displaystyle
\frac{4^s\Gamma(1+s)}{\pi\|m\|_2^{2+2s}}$.

It can be proved that for $\mathit{II}$, $\mathit{III}$ and $\mathit{IV}$ (the error terms) we have better decay in~$m$.
\end{proof}

\begin{rem}
Lemma~\ref{lem:expressionKs2d} can also be proved for any dimension $n>2$. The proofs of these statements need rather cumbersome computations, inspired by \cite[Subsection 6.4]{Chinta-Jorgenson-Karlsson}, with suitable modifications, and we omit them for the sake of clarity.
\end{rem}

\begin{rem}
In order to prove analogous ``asymptotic estimates'' for the kernel of the fractional discrete integral by using Mellin transforms, in two dimensions, one ends up to an expression of the form
\begin{multline*}
\int_0^{\infty}e^{-4t}I_{m_1}(2t)I_{m_2}(2t)
\frac{dt}{t^{1-s}} \\
= \frac{4^s}{\pi}\frac{1}{2\pi i} \int_{c_1-i\infty}^{c_1+i\infty} \Gamma(1/2-x_1)\Gamma(1/2+x_1-s)
\frac{\Gamma(m_1+x_1)}{\Gamma(m_1-x_1+1)}
\frac{\Gamma(m_2-x_1+s)}{\Gamma(m_2+x_1-s+1)} \,dx_1.
\end{multline*}
There is a discussion about this identity when $m_1=m_2=0$ in \cite[Section~6.2]{Chinta-Jorgenson-Karlsson}
that can be of interest. Thus in a similar way to obtain
\begin{equation}
\label{eq:fractionalIntegralKernelEstimate}
0<K_{-s}(m) = \frac{c_{2,-s}}{\|m\|_2^{2-2s}}+\text{higher order terms},\quad
\text{when both } m_i\to \pm\infty, \quad  i=1,2.
\end{equation}
\end{rem}

\section{Illustrations in dimension one}
\label{sec:oneDim}

As we said in the Introduction, 
we may raise the question whether the Poisson problem $(-\Delta)^su=f$ can be discretized by using
the formulas we obtained for the fractional discrete Laplacian and fractional discrete integral.
One would like to see if the solutions to the discrete Poisson problem converge
in some sense to the solutions of the continuous Poisson problem.

Some numerical experiments by using a numerical method were already made in \cite{Huang-Oberman}. Our intention is not to make numerical experiments here, but to illustrate the convergence of the discretized problem to the continuous one. Furthermore, with the analogous formula in one dimension for the fractional discrete integral kernel given in \eqref{eq:kernelFrIntegralOned} we are also able to \textit{solve} the discrete Poisson problem, so we improve in a sense the examples performed in~\cite{Huang-Oberman}.

Formula \eqref{eq:fractionalKernelOned} provides an exact expression for the kernel of the fractional discrete Laplacian in dimension one. This and the asymptotics in that formula will allow us to compute and draw some approximations in one dimension.

We present some pictures of known examples in dimension one. Some of the examples included in this section and in Section~\ref{sec:twoDim} were studied, with different numerical methods, by Huang and Oberman in~\cite{Huang-Oberman}.

By using~\eqref{eq:fractionalKernelOned} and \eqref{eq:constante1d}, and by taking into account \eqref{eq:term2} we can write
\[
(-\Delta_h)^su_{j} \approx F_1(j,h,s)+F_2(j,h,s)+F_3(j,h,s),\quad j\in \Z,
\]
where
\begin{equation}
\label{eq:F1}
F_1(j,h,s) := \frac{c_{s}}{h^{2s}} \sum_{\substack{|m|\le N\\ m\ne 0}} (u_{j}-u_{j-m})
\frac{\Gamma(|m|-s)}{\Gamma(|m|+1+s)},
\end{equation}
\begin{equation}
\label{eq:F2}
F_2(j,h,s) := \frac{c_{s}}{h^{2s}} \, u_{j} \sum_{|m|\ge N} \frac{1}{|m|^{1+2s}}
\end{equation}
and
\begin{equation*}
F_3(j,h,s) := \frac{c_{s}}{h^{2s}} \sum_{|m|\ge N}u_{j-m} \frac{1}{|m|^{1+2s}},
\end{equation*}
for certain $N\in \N$.

On the other hand, formula~\eqref{eq:kernelFrIntegralOned} allows us to \textit{solve} the discrete Poisson problem
for a given datum $f$. By using \eqref{eq:kernelFrIntegralOned} and again noticing~\eqref{eq:term2} we can write
\[
u_{j}\approx(-\Delta_h)^{-s}f_{j}=U_1(j,h,s)+U_2(j,h,s),\quad j\in \Z,
\]
where
\begin{equation}
\label{eq:U1}
U_1(j,h,s) := c_{-s}h^{2s} \sum_{|m|\le N}f_{j-m} \frac{\Gamma(|m|+s)}{\Gamma(|m|+1-s)}
\end{equation}
and
\begin{equation*}
U_2(j,h,s) := c_{-s}h^{2s} \sum_{|m|\ge N}f_{j-m} \frac{1}{|m|^{1-2s}},
\end{equation*}
for certain $N\in \N$.

\subsection{Example 1.}

The first example to be examined is the function
\begin{equation}
\label{eq:u1}
u(x)=e^{-x^2}.
\end{equation}
The fractional Laplacian of $u$ at $x=0$ can be obtained exactly by using Fourier transform,
\begin{equation}
\label{eq:primerEjemplo}
(-\Delta)^su(0) = \frac{1}{\sqrt{\pi}} \int_0^{\infty}k^{2s}e^{-k^2/4}\,dk
= \frac{4^s\Gamma(1/2+s)}{\sqrt{\pi}},
\end{equation}
see \cite[Section~6.1]{Huang-Oberman}.
Since the function $u(x)$ in \eqref{eq:u1} has rapid decay, we may ignore the term $F_3(j,h,s)$.
In Figure~\ref{figure1} we can see the exact value of \eqref{eq:primerEjemplo} and the approximation $F_1(j,h,s)+F_2(j,h,s)$ in \eqref{eq:F1} and \eqref{eq:F2} for the fractional discrete Laplacian related to $r_hu$.

\begin{figure}[h]
\includegraphics[scale=0.5]{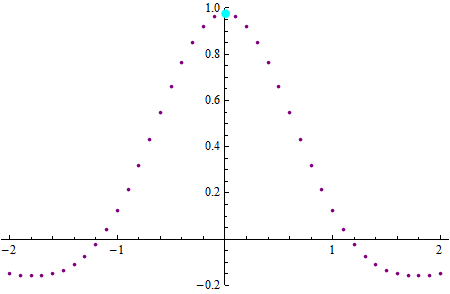}
\caption{The small-dotted line represents $F_1(j,h,s)+F_2(j,h,s)$ for the function $u(x)$ in~\eqref{eq:u1}, with $N=1000$, $s=0.25$, $h=0.1$ and $j\in \Z$, $-20\le j\le 20$, so that the horizontal axis is $[-20h,20h]$. The blue point is the value in~\eqref{eq:primerEjemplo}.}
\label{figure1}
\end{figure}

\subsection{Example 2.}
The second example is the function
\begin{equation}
\label{eq:u2}
u(x)=(1+x^2)^{-(1/2-s)}
\end{equation}
with the exact fractional Laplacian
\begin{equation}
\label{eq:segundo-ejemplo}
(-\Delta)^su(x) = \frac{4^s\Gamma(1/2+s)}{\Gamma(1/2-s)}(1+x^2)^{-(1/2+s)} =: f(x),
\end{equation}
see \cite[formula~(41)]{Huang-Oberman}.

The function $u(x)$ decays algebraically, and we are going to ignore the term $F_3(j,h,s)$ again. Actually, it is possible to take into account that this decay is as $(hm)^{-(1-2s)}$ to estimate a value for $F_3(j,h,s)$. Some comments and strategies about this are given in \cite{Huang-Oberman}. But our goal in this section is to illustrate our results, so we will not worry about it.
Moreover, we are using $N=1000$ so $F_3$ is not as important as it would be for small values of~$N$.
In Figure~\ref{figure2} the continuous line represents $f(x)$ in \eqref{eq:segundo-ejemplo} and the dotted line is the approximation $F_1(j,h,s)+F_2(j,h,s)$ in \eqref{eq:F1} and \eqref{eq:F2} for the fractional
discrete Laplacian of $r_hu$, with $u$ as in~\eqref{eq:u2}.

\begin{figure}[H]
\includegraphics[scale=0.5]{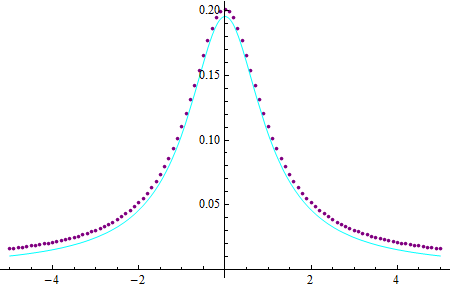}
\caption{The dotted line represents $F_1(j,h,s)+F_2(j,h,s)$ for the restriction of the function $u(x)$ in \eqref{eq:u2}, with $N=1000$, $s=0.4$, $h=0.1$ and $j\in \Z$, $-50\le j\le 50$, so that the horizontal axis is $[-50h,50h]$. The continuous line is $f(x)$ in~\eqref{eq:segundo-ejemplo}.}
\label{figure2}
\end{figure}

For the discrete Poisson problem, in Figure~\ref{figure3}
we plot $U_1(j,h,s)$ in \eqref{eq:U1} for $r_hf$, with $f(x)$ as in \eqref{eq:segundo-ejemplo}. We omit the term $U_2(j,h,s)$ because of the decay of the function $f(x)$ (we again remit to \cite{Huang-Oberman} for strategies to take into account this term). The continuous line is $u(x)$ in~\eqref{eq:u2}.

\begin{figure}[H]
\includegraphics[scale=0.5]{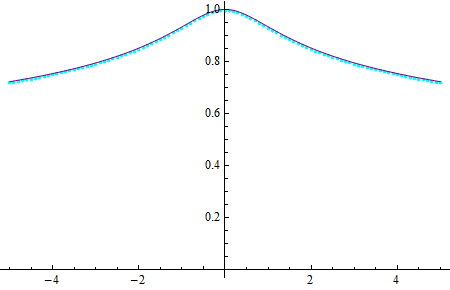}
\caption{The dotted line represents $U_1(j,h,s)$ for the function \eqref{eq:segundo-ejemplo}, with $N=1000$, $s=0.4$, $h=0.1$ and $j\in \Z$, $-50\le j\le 50$, so that the horizontal axis is $[-50h,50h]$. The continuous line is $u(x)$ in~\eqref{eq:u2}.}
\label{figure3}
\end{figure}

\subsection{A preliminary lemma.}
Before continuing with the next examples, we first prove the following lemma, see also \cite[Lemma 2.4]{Biler-Imbert-Karch} (note that there are some typos therein). Analogous computations, but for the fractional Laplacian instead of the fractional integral, are carried out in \cite{Getoor}. Nevertheless, we include our own proof for completeness. Since the lemma is stated for any dimension $n\ge1$, we will denote $\|x\|_2:=(\sum_{i=1}^n|x_i|^2)^{1/2}$, for $x\in \R^n$.

\begin{lem}
\label{lem:uBiler}
Let $0<s<1$, $\gamma>0$ and $x\in \R^n$. Let
\[
f(x)=(1-\|x\|_2^2)_+^{\gamma/2},
\]
where we use the notation $y_+ = \max\{y,0\}$.
Then
\begin{equation*}
u(x) := (-\Delta)^{-s}f(x) = \begin{cases} C_{\gamma,s,n}
\,{}_2F_1\Big(\frac{n-2s}{2},-\frac{\gamma+2s}{2};\frac{n}{2};\|x\|_2^2\Big),& \text{ if }\, \|x\|_2\le 1, \\
\widetilde{C}_{\gamma,s,n}\|x\|_2^{2s-n}
\,{}_2 F_1\Big(\frac{n-2s}{2},\frac{2-2s}{2};\frac{n+\gamma}{2}+1;\frac{1}{\|x\|_2^2}\Big),& \text{ if }\, \|x\|_2\ge 1, \end{cases}
\end{equation*}
where
\[
C_{\gamma,s,n} = 2^{-2s} \,\frac{\Gamma\big(\frac{n-2s}{2}\big)\Gamma\big(\frac{\gamma}{2}+1\big)}
{\Gamma\big(\frac{2s+\gamma}{2}+1\big)\Gamma\big(\frac{n}{2}\big)},
\]
and
\[
\widetilde{C}_{\gamma,s,n} = 2^{-2s} \,\frac{\Gamma\big(\frac{n-2s}{2}\big)
\Gamma\big(\frac{\gamma}{2}+1\big)}
{\Gamma\big(\frac{n+\gamma}{2}+1\big)\Gamma(s)},
\]
and ${}_2F_1$ is the Gaussian or ordinary hypergeometric function.
\end{lem}

\begin{proof}
On one hand, it is known that (see \cite[p.~171]{StWe})
\[
\mathcal{F}f(\xi) = \pi^{-\gamma/2} \Gamma\Big(\frac{\gamma}{2}+1\Big) \|\xi\|_2^{-\frac{n+\gamma}{2}}
J_{(n+\gamma)/2}(2\pi\|\xi\|_2),
\]
where $J_\alpha$ is the Bessel function of order $\alpha$. On the other hand (see for example \cite[Ch.~V.1]{StSingular}),
\[
u(x) = (-\Delta)^{-s}f(x) = (2\pi)^{-2s}\mathcal{F}^{-1}(\|\cdot\|_2^{-2s}\mathcal{F}f)(x).
\]
Since $f$ is a radial function, the Fourier transform of $f$ is in fact the so-called Hankel transform $\mathcal{H}f$, and $\mathcal{H}^{-1}=\mathcal{H}$. Therefore (see \cite[p.~155]{StWe}),
\begin{equation}
\label{eq:uBiler}
u(x) = \frac{\pi^{-2s+1-\frac{\gamma}{2}}}{2^{2s-1} \|x\|_2^{\frac{n}{2}-1}}\Gamma\Big(\frac{\gamma}{2}+1\Big)
\int_0^\infty\|\xi\|_2^{-2s-\frac{\gamma}{2}}
J_{(n+\gamma)/2}(2\pi\|\xi\|_2) J_{n/2-1}(2\pi\|\xi\|_2 \|x\|_2)\,d\|\xi\|_2.
\end{equation}
We will use the identities (see \cite[Ch.~13.4, p.~401]{Watson})
\begin{equation}
\label{eq:Watson1}
\begin{aligned}
\int_0^{\infty} \frac{J_{\mu}(at)J_{\nu}(bt)}{t^{\lambda}}\,dt
&= \frac{b^\nu\Gamma\big(\frac{\mu+\nu-\lambda+1}{2}\big)}{2^{\lambda}a^{\nu-\lambda+1}
\Gamma(\nu+1)\Gamma\big(\frac{\lambda+\mu-\nu+1}{2}\big)} \\
&\qquad \times {}_2F_1\Big(\frac{\mu+\nu-\lambda+1}{2},\frac{\nu-\lambda-\mu+1}{2};\nu+1;\frac{b^2}{a^2}\Big), \quad 0<b<a,
\end{aligned}
\end{equation}
and
\begin{equation}
\label{eq:Watson2}
\begin{aligned}
\int_0^{\infty} \frac{J_{\mu}(at)J_{\nu}(bt)}{t^{\lambda}}\,dt
&= \frac{a^\mu\Gamma\big(\frac{\mu+\nu-\lambda+1}{2}\big)}{2^{\lambda}b^{\mu-\lambda+1}
\Gamma(\mu+1)\Gamma\big(\frac{\lambda+\nu-\mu+1}{2}\big)} \\
&\qquad \times {}_2F_1\Big(\frac{\mu+\nu-\lambda+1}{2},\frac{\mu-\lambda-\nu+1}{2};\mu+1;\frac{a^2}{b^2}\Big), \quad 0<a<b,
\end{aligned}
\end{equation}
which are valid when $\mu+\nu-\lambda>-1$ and $\lambda>-1$. Now, by taking $\mu=\tfrac{n+\gamma}{2}$, $\nu=\tfrac{n}{2}-1$, $\lambda=2s+\tfrac{\gamma}{2}$, $a=2\pi$ and $b=2\pi\|x\|_2$, we apply \eqref{eq:Watson1} in the case $\|x\|_2\le1$ and~\eqref{eq:Watson2} in the case $\|x\|_2>1$ to compute the integral in~\eqref{eq:uBiler}. The desired result follows.
\end{proof}

\subsection{Example 3.}
The following example involves the function with compact support
\begin{equation}
\label{eq:f2}
f(x)=(1-x^2)_+^{1-s}.
\end{equation}
By Lemma~\ref{lem:uBiler}, the solution of $(-\Delta)^su=f$ is given by
\begin{equation}
\label{eq:tercer-ejemplo}
u(x) = \begin{cases} 4^{-s} \frac{\Gamma(1/2-s)\Gamma(2-s)}{\pi^{1/2}}(1-(1-2s)x^2),& \text{ if }\, |x|\le 1, \\
4^{-s} \frac{\Gamma(1/2-s)\Gamma(2-s)}{\Gamma(s)\Gamma(5/2-s)}|x|^{2s-1}
\,{}_2F_1\Big(1/2-s,1-s;5/2-s;\frac{1}{|x|^2}\Big),& \text{ if }\, |x|\ge 1, \end{cases}
\end{equation}
see also \cite{Biler-Imbert-Karch} and \cite[(42)]{Huang-Oberman}.

In Figure~\ref{figure4} the dotted line represents the sum $F_1(j,h,s)+F_2(j,h,s)$ in \eqref{eq:F1} and \eqref{eq:F2} related to the function $u(x)$ in~\eqref{eq:tercer-ejemplo}.

\begin{figure}[H]
\includegraphics[scale=0.5]{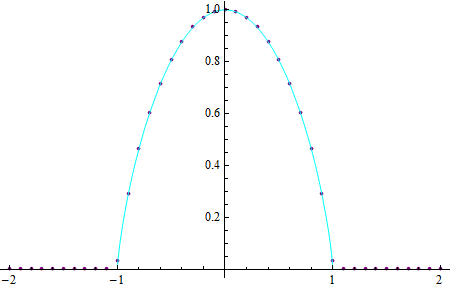}
\caption{The dotted line represents $F_1(j,h,s)+F_2(j,h,s)$ for the function \eqref{eq:tercer-ejemplo}, with $N=1000$, $s=0.25$, $h=0.1$ and $j\in \Z$, $-20\le j\le 20$, so that the horizontal axis is $[-20h,20h]$. The continuous line is $f(x)$ in~\eqref{eq:f2}.}
\label{figure4}
\end{figure}

The continuous line is $f(x)$ in~\eqref{eq:f2}.
For solving the discrete Poisson problem in this case, in Figure~\ref{figure5} the dotted line represents $U_1(j,h,s)$ in \eqref{eq:U1} for the function $f(x)$ in~\eqref{eq:f2}. Since $f(x)$ has compact support, we can choose $N$ such that the term $U_2(j,h,s)$ is zero. The continuous line is $f(x)$ in~\eqref{eq:tercer-ejemplo}.

\begin{figure}[H]
\includegraphics[scale=0.5]{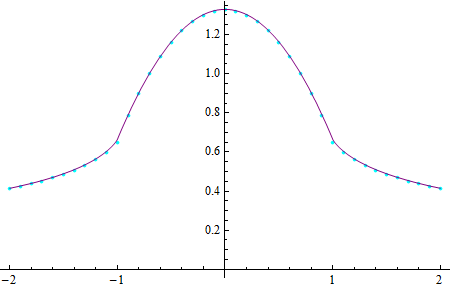}
\caption{The dotted line represents $U_1(j,h,s)$ for the function \eqref{eq:f2}, with $N=20$, $s=0.25$, $h=0.1$ and $j\in \Z$, $-20\le j\le 20$, so that the horizontal axis is $[-20h,20h]$. The continuous line is $u(x)$ in~\eqref{eq:tercer-ejemplo}.}
\label{figure5}
\end{figure}

\subsection{Example 4.}
We present another function with compact support:
\begin{equation}
\label{eq:f3}
f(x) = (1-x^2)_+^{2-s}.
\end{equation}
By Lemma~\ref{lem:uBiler}, the solution of $(-\Delta)^su=f$ is given by
\begin{equation}
\label{eq:cuarto-ejemplo}
u(x) = \begin{cases} 4^{-s} \frac{\Gamma(1/2-s)\Gamma(3-s)}{\pi^{1/2}}
\Big(1-(2-4s)x^2+\big(1-\frac{8}{3}s+\frac{4}{3}s^2\big)x^4\Big),& \text{ if }\, |x|\le 1, \\
4^{-s} \frac{\Gamma(1/2-s)\Gamma(3-s)}{\Gamma(s)\Gamma(7/2-s)}|x|^{2s-1}
\,{}_2F_1\Big(1/2-s,1-s;7/2-s;\frac{1}{|x|^2}\Big),& \text{ if }\, |x|\ge 1, \end{cases}
\end{equation}
see also \cite{Biler-Imbert-Karch} and \cite[formula~(43)]{Huang-Oberman}.

In Figure~\ref{figure6} the dotted line is $F_1(j,h,s)+F_2(h,j,s)$ in~\eqref{eq:F1} and~\eqref{eq:F2} for $u(x)$ as in~\eqref{eq:cuarto-ejemplo}. The continuous line is the exact $F(x)$ in~\eqref{eq:f3}.

\begin{figure}[H]
\includegraphics[scale=0.5]{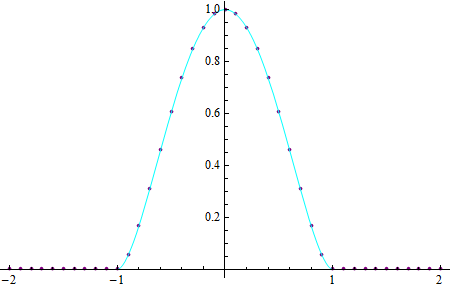}
\caption{The dotted line represents $F_1(j,h,s)+F_2(j,h,s)$ for the function~\eqref{eq:cuarto-ejemplo}, with $N=1000$, $s=0.25$, $h=0.1$ and $j\in \Z$, $-20\le j\le 20$, so that the horizontal axis is $[-20h,20h]$. The continuous line is $f(x)$ in~\eqref{eq:f3}.}
\label{figure6}
\end{figure}

For the discrete Poisson problem, in Figure~\ref{figure7} the dotted line is $U_1(j,h,s)$ in~\eqref{eq:U1} for~\eqref{eq:f3}. Again, since $f(x)$ has compact support, we can choose $N$ large such that $U_2(j,h,x)$ is zero. The continuous line is $u(x)$ in~\eqref{eq:cuarto-ejemplo}.

\begin{figure}[H]
\includegraphics[scale=0.5]{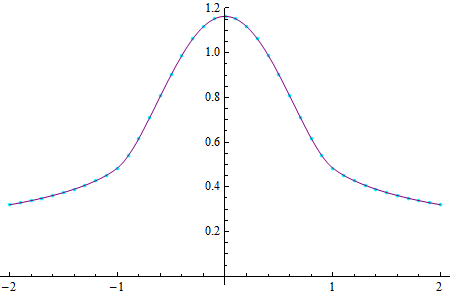}
\caption{The dotted line represents $U_1(j,h,s)$ for the function $f(x)$ in~\eqref{eq:f3}, with $N=20$, $s=0.25$, $h=0.1$ and $j\in \Z$, $-20\le j\le 20$, so that the horizontal axis is $[-20h,20h]$. The continuous line is $u(x)$ in~\eqref{eq:cuarto-ejemplo}.}
\label{figure7}
\end{figure}

\section{Illustrations in dimension two}
\label{sec:twoDim}

In this section we plot approximations of some known examples in dimension two.

Unlike the one dimensional case, we do not have explicit expressions for the two dimensional kernels of the fractional discrete Laplacian and discrete integral. Therefore, computation and evaluation of the kernels is complicated, so we are going to use a less precise fast method. We only have at our disposal the asymptotic estimates for the kernels $K_s(m)$ and $K_{-s}(m)$ of Lemma~\ref{lem:expressionKs2d} and Theorem~\ref{lem:kernelFractionalIntegral}.
Nevertheless, in view of the one dimensional case and Lemma~\ref{eq:lowerBound}, we conjecture that we can approximate the kernel of the fractional discrete Laplacian by the main term in~\eqref{eq:frKernelEstimate2d}, and the kernel of the fractional discrete integral by the main term in~\eqref{eq:fractionalIntegralKernelEstimate}. In this way, for the fractional discrete Laplacian we take
\begin{equation}
\label{eq:Fdos}
(-\Delta_h)^s u_{j} = \frac{c_{2,s}}{h^{2s}} \sum_{\substack{m\in \Z^2\setminus\{0\}\\ |m_i|\le N}}
\frac{u_{j}-u_{j-m}}{\|m\|_2^{2+2s}}+E =: F(j,h,s)+E,
\quad j\in \Z^2,
\end{equation}
for certain $N\in \N$, where $E$ involves the sum with the error terms.

On the other hand, we use formula \eqref{eq:fractionalIntegralKernelEstimate} as an approximation that allows us to \textit{solve} the discrete Poisson problem for a given datum $f$. We write
\begin{equation}
\label{eq:Udos}
\begin{aligned}
u_{j} = (-\Delta_h)^{-s}f_{j} &= c_{-2,s}h^{2s}\bigg(f_{j}K_{-s}(0)
+ \sum_{\substack{m\in \Z^2\setminus\{0\}\\ |m_i|\le N}} \frac{f_{j-m} }{\|m\|_2^{2-2s}}\bigg)+\widetilde{E} \\
&=: U(j,h,s)+\widetilde{E}, \quad j\in \Z^2,
\end{aligned}
\end{equation}
for certain $N\in \N$, where $\widetilde{E}$ involves the sum with the error terms. Moreover, $K_{-s}(0)$ in \eqref{eq:Udos} can be explicitly evaluated. Indeed,
\begin{equation*}
K_{-s}(0)
= \frac{1}{\Gamma(s)} \int_0^\infty G(0,t) \frac{dt}{t^{1-s}}
= \frac{1}{\Gamma(s)} \int_0^\infty e^{-4t}{(I_0(2t))}^2 \,\frac{dt}{t^{1-s}}
= 4^{-s} \,{}_3F_2\Big(\frac12,\frac{1+s}2,\frac{s}2;1,1;1\Big),
\end{equation*}
where ${}_3F_2$ is the generalized hypergeometric function.

\subsection{Example 1.}
The first example we deal with is a function with compact support:
\begin{equation}
\label{eq:f2-2dim}
f(x)=(1-\|x\|_2^2)_+^{1-s}, \quad x\in \R^2.
\end{equation}
By Lemma~\ref{lem:uBiler}, the solution of $(-\Delta)^su=f$ is given by
\begin{equation}
\label{eq:tercer-ejemplo-2d}
u(x)=\begin{cases}4^{-s}\Gamma(1-s)\Gamma(2-s)(1-(1-s)\|x\|_2^2),& \text{ if }\, \|x\|_2\le 1, \\
4^{-s} \frac{\Gamma(1-s)\Gamma(2-s)}{\Gamma(s)\Gamma(3-s)}\|x\|_2^{2s-2}
\,{}_2F_1\Big(1-s,1-s;3-s;\frac{1}{\|x\|_2^2}\Big),& \text{ if }\, \|x\|_2\ge 1. \end{cases}
\end{equation}
In Figure~\ref{figure8} we plot the exact function $f(x)$ in~\eqref{eq:f2-2dim} on the left, the approximation $F(j,h,s)$ in~\eqref{eq:Fdos} related to~\eqref{eq:tercer-ejemplo-2d} in the center and the error (difference) on the right.

\begin{figure}[H]
\includegraphics[scale=0.65]{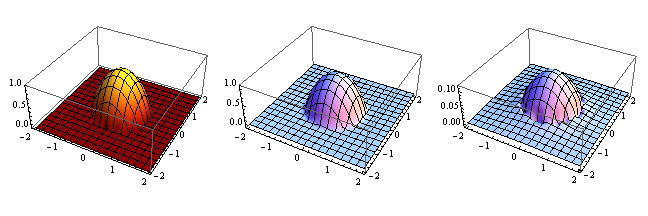}
\caption{The surface on the left is $f(x)$ in~\eqref{eq:f2-2dim}, with $x\in[-2,2]\times[-2,2]$. In the center we find the approximation $F(j,h,s)$ related to the function $u(x)$ in~\eqref{eq:tercer-ejemplo-2d}, with $N=500$, $s=0.25$, $h=0.1$ and $j=(j_1,j_2)\in \Z^2$, $-20\le j_i\le 20$. The mesh is $[-20h,20h]\times[-20h,20h]$. The error in each point of the considered mesh is plotted on the right.}
\label{figure8}
\end{figure}

In Figure~\ref{figure9} we plot the exact function $u(x)$ of \eqref{eq:tercer-ejemplo-2d} on the left, and the picture arising from $U(j,h,s)$ in~\eqref{eq:Udos} related to~\eqref{eq:f2-2dim} on the center. This situation corresponds with the discrete Poisson problem. The error is drawn on the right. Observe that, since $f$ is a function with compact support, there is only a finite number of nonzero terms in the computation of $U(j,h,s)$.

\begin{figure}[H]
\includegraphics[scale=0.65]{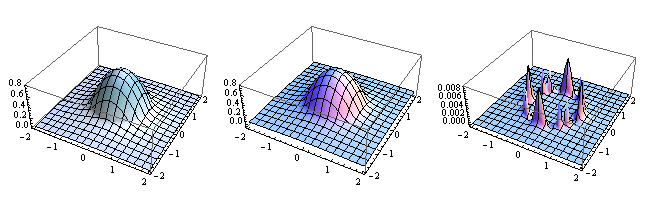}
\caption{The surface on the left is $u(x)$ in~\eqref{eq:tercer-ejemplo-2d}, with $x\in[-2,2]\times[-2,2]$. In the center we find the approximation $U(j,h,s)$ related to $f(x)$ in~\eqref{eq:f2-2dim}, with $N=40$, $s=0.25$, $h=0.1$ and $j=(j_1,j_2)\in \Z^2$, $-20\le j_i\le 20$. The mesh is $[-20h,20h]\times[-20h,20h]$. The error in each point of the considered mesh is plotted on the right.}
\label{figure9}
\end{figure}

\subsection{Example 2.}
The second example concerns the compactly supported function
\begin{equation}
\label{eq:f3-2dim}
f(x)=(1-\|x\|_2^2)_+^{2-s}, \quad x\in \R^2.
\end{equation}
By Lemma~\ref{lem:uBiler}, the solution to $(-\Delta)^su=f$ is given by
\begin{equation}
\label{eq:tercer-ejemplo-2d-bis}
u(x) = \begin{cases}4^{-s}2^{-1}\Gamma(1-s)\Gamma(3-s)(1-(2-2s)\|x\|_2^2
+\big(1-\frac32s+\frac12s^2\big)\|x\|_2^4),& \text{ if }\, \|x\|_2\le 1, \\
4^{-s} \frac{\Gamma(1-s)\Gamma(2-s)}{\Gamma(s)\Gamma(3-s)}\|x\|_2^{2s-2}
\,{}_2F_1\Big(1-s,1-s;3-s;\frac{1}{\|x\|_2^2}\Big),& \text{ if }\, \|x\|_2\ge 1.
\end{cases}
\end{equation}
In Figure~\ref{figure10} we see the picture of $f(x)$ of \eqref{eq:f3-2dim} on the left, the approximation $F(j,h,s)$ in~\eqref{eq:Fdos} related to \eqref{eq:tercer-ejemplo-2d-bis} in the center, and the error on the right.

\begin{figure}[H]
\includegraphics[scale=0.65]{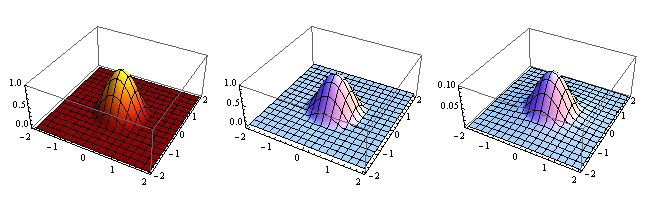}
\caption{The surface on the left is $f(x)$ in~\eqref{eq:f3-2dim}, with $x\in[-2,2]\times[-2,2]$. In the center we find the approximation $F(j,h,s)$ related to $u(x)$ in \eqref{eq:tercer-ejemplo-2d-bis}, with $N=500$, $s=0.25$, $h=0.1$ and $j=(j_1,j_2)\in \Z^2$, $-20\le j_i\le 20$. The mesh is $[-20h,20h]\times[-20h,20h]$. The error in each point of the considered mesh is plotted on the right.}
\label{figure10}
\end{figure}

Next we illustrate the problem of solving the discrete Poisson problem in this case. In Figure~\ref{figure11} the picture on the left is $u(x)$ given in~\eqref{eq:tercer-ejemplo-2d-bis}, the approximation $U(j,h,s)$ in \eqref{eq:Udos} related to $f(x)$ in \eqref{eq:f3-2dim} is in the center (observe that there are finite summands, so we can choose $N$ large so that the term $\widetilde{E}$ is zero), and the error is on the right.

\begin{figure}[H]
\includegraphics[scale=0.65]{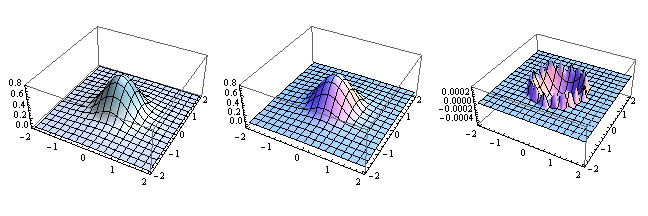}
\caption{The surface on the left is $u(x)$ in~\eqref{eq:tercer-ejemplo-2d-bis}, with $x\in[-2,2]\times[-2,2]$. In the center we find the approximation $U(j,h,s)$ related to $f(x)$ in \eqref{eq:f3-2dim}, with $N=40$, $s=0.25$, $h=0.1$ and $j=(j_1,j_2)\in \Z^2$, $-20\le j_i\le 20$. The mesh is $[-20h,20h]\times[-20h,20h]$. The error in each point of the considered mesh is plotted on the right.}
\label{figure11}
\end{figure}

\subsection{Example 3.}
The last function we take into account is the following. Let $0<s<1$ and $0<\alpha<2-2s$. Consider the function
\begin{equation}
\label{eq:u4-2dim}
u(x)=\|x\|_2^{-\alpha}, \quad x\in \R^2.
\end{equation}
It is known that
\begin{equation}
\label{eq:ejemploRos}
(-\Delta)^su(x) = 2^{2s}\, \frac{\Gamma(\alpha/2+s)\Gamma(1-\alpha/2)}
{\Gamma(1-\alpha/2-s)\Gamma(\alpha/2)}\|x\|_2^{-\alpha-2s} =: f(x).
\end{equation}
Observe that $f(x)$ in \eqref{eq:ejemploRos} has a singularity at the origin. However, we are going to check the behavior of the numerical method by computing the corresponding approximations $F(j,h,s)$ in \eqref{eq:Fdos} for the fractional discrete Laplacian and $U(j,h,s)$ in \eqref{eq:Udos} for the fractional discrete integral. In order to avoid the singularity in the origin, we move the mesh a distance $h/2$ in the direction of each
coordinate axis, so we consider a mesh $[-(J+1/2)h,(J+1/2)h]\times[-(J+1/2)h,(J+1/2)h]$. Then, we see in Figures~\ref{figure12} and~\ref{figure13} respectively that these approximations are close to the exact functions $f(x)$ in \eqref{eq:ejemploRos} and $u(x)$ in~\eqref{eq:u4-2dim}, and the error of the approximation is large only near the origin, which is expected because both $u$ and $f$ tend to infinity close to zero.

\begin{figure}[H]
\includegraphics[scale=0.65]{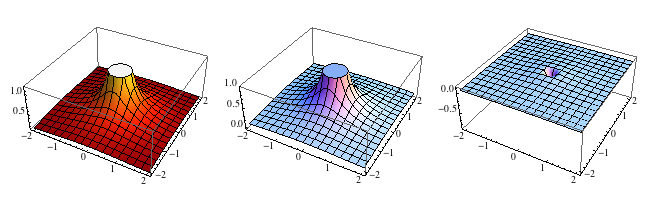}
\caption{The surface on the left if $f(x)$ in \eqref{eq:ejemploRos} with $\alpha=0.5$ and $s=0.3$, for $x\in[-2,2]\times[-2,2]$. In the center we find the approximation $F(j,h,s)$ related to $u(x)$ in \eqref{eq:u4-2dim}, with $N=500$, $h=0.1$ and $j=(j_1,j_2)\in \Z^2$, $-20\le j_i\le 20$. The mesh is $[-(20+1/2)h,(20+1/2)h]
\times[-(20+1/2)h,(20+1/2)h]$. The error in each point of the considered mesh is plotted on the right.}
\label{figure12}
\end{figure}

\begin{figure}[H]
\includegraphics[scale=0.65]{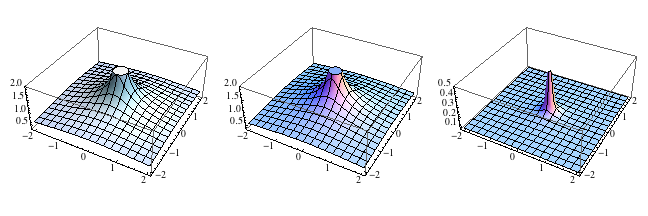}
\caption{The surface on the left is $u(x)$ in \eqref{eq:u4-2dim} with $\alpha=0.5$, for $x\in[-2,2]\times[-2,2]$. In the center we find the approximation $U(j,h,s)$ related to $f(x)$ in~\eqref{eq:ejemploRos}, with $N=500$, $s=0.3$, $h=0.1$ and $j=(j_1,j_2)\in \Z^2$, $-20\le j_i\le 20$. The mesh is $[-(20+1/2)h,(20+1/2)h]\times[-(20+1/2)h,(20+1/2)h]$. The error in each point of the considered mesh is plotted on the right.}
\label{figure13}
\end{figure}

\section{Technical lemmas and properties of Bessel functions}
\label{sec:technical}

\subsection{Some technical lemmas}
\label{subsec:technical-lemmas}

Lemmas in this subsection are needed in the proof of Theorem~\ref{thm:consistencia1d}. They are also useful to get estimates for the kernels of the fractional discrete Laplacian \eqref{eq:kernelFractionalLaplacian} and for the fractional integral kernel in Theorem~\ref{lem:kernelFractionalIntegral}.

The following lemma is elementary. We include it for the sake of completeness.

\begin{lem}
\label{lem:Mario}
Let $\lambda>0$. Let $a,b$ be real numbers such that $0\le a<b<\infty$. Then
\[
\min\{\lambda,1\}\le\frac{b^{\lambda}-a^{\lambda}}{b^{\lambda-1}(b-a)}\le \max\{\lambda,1\}.
\]
\end{lem}

\begin{proof}
Let us first suppose that $\lambda\ge1$. Then
\[
0\le a<b<\infty\Rightarrow 0\le a/b<1\Rightarrow 0\le(a/b)^{\lambda}
\le a/b<1\Rightarrow \frac{b^{\lambda}-a^{\lambda}}{b^{\lambda-1}(b-a)}=\frac{1-(a/b)^{\lambda}}{1-a/b}\ge1.
\]
On the other hand, by applying the mean value theorem to the function $x^{\lambda}$, we get
\[
\frac{b^{\lambda}-a^{\lambda}}{b^{\lambda-1}(b-a)}=\frac{1-(a/b)^{\lambda}}{1-a/b}=\lambda x^{\lambda-1}\le \lambda,
\]
for certain $x\in (a/b,1)$. In the case $0<\lambda<1$, the proof is analogous.
\end{proof}

\begin{lem}
\label{eq:lowerBound}
Let $0<s<1$, $t\in \R$, and $m\in \Z$, $m\neq 0$. Then,
\begin{equation}
\label{eq:lowerLap}
\bigg|\frac{\Gamma(|m|-s)}{\Gamma(|m|+1+s)}-\frac{1}{|m|^{1+2s}}\bigg|\le \frac{C_{s}}{|m|^{2+2s}}.
\end{equation}
Also, if $0<s<1/2$, we have
\begin{equation}
\label{eq:lowerFrac}
\bigg|\frac{\Gamma(|m|+s)}{\Gamma(|m|+1-s)}-\frac{1}{|m|^{1-2s}}\bigg|\le \frac{C_{s}}{|m|^{2-2s}}.
\end{equation}
\end{lem}

\begin{proof}
Let us begin with the proof of~\eqref{eq:lowerLap}. Without loss of generality, take $m\in \N$. We write
\begin{align*}
\bigg|\frac{\Gamma(m-s)}{\Gamma(m+1+s)}-\frac{1}{m^{1+2s}}\bigg|
&\le\bigg|\frac{\Gamma(m-s)}{\Gamma(m+1+s)}-\frac{1}{(m-s)^{1+2s}}\bigg| +\bigg|\frac{1}{(m-s)^{1+2s}}-\frac{1}{m^{1+2s}}\bigg|.
\end{align*}
The second term can be easily estimated, just by applying Lemma~\ref{lem:Mario} with $\lambda=1+2s$, $a=\frac{1}{m}$ and $b=\frac{1}{m-s}$, namely
\[
\bigg|\frac{1}{(m-s)^{1+2s}}-\frac{1}{m^{1+2s}}\bigg|
\simeq \frac{1}{(m-s)^{2s}}\bigg(\frac{1}{m-s}-\frac{1}{m}\bigg) \simeq \frac{C_{s}}{m^{2+2s}},
\]
where the symbol $\simeq$ means that constants depend only on~$s$.
Now we study the first term. For $k\in \N$, we have (see for instance \cite[Section~7]{Tricomi-Erdelyi})
\begin{equation}
\label{eq:Tricomi}
\frac{\Gamma(k-s)}{\Gamma(k+n+s)}=\frac{1}{\Gamma(n+2s)}\int_0^{\infty}e^{-(k-s)v}(1-e^{-v})^{n+2s-1}\,dv.
\end{equation}
With this,
\begin{align*}
\Gamma(1+2s)\bigg|\frac{\Gamma(m-s)}{\Gamma(m+n+s)}-\frac{1}{(m-s)^{1+2s}}\bigg|
&\le \int_0^{\infty}e^{-(m-s)v}\big|v^{2s}-(1-e^{-v})^{2s}\big|\,dv\\
&=\int_0^{\infty}e^{-(m-s)v}v^{2s}\bigg|1-\Big(\frac{1-e^{-v}}{v}\Big)^{2s}\bigg|\,dv\\
&\simeq \int_0^{\infty}e^{-(m-s)v}v^{2s}\bigg|1-\frac{1-e^{-v}}{v}\bigg|\,dv\\\
&\le \frac12\int_0^{\infty}e^{-(m-s)v}v^{2s+1}\,dv\simeq\frac{\Gamma(1+2s)}{2}\frac{1}{m^{2+2s}},
\end{align*}
where we applied Lemma~\ref{lem:Mario}, 
and in the last inequality we used that $\frac{v^2}{2}>v-1+e^{-v}$ for $v\in (0,\infty)$.
The proof of \eqref{eq:lowerFrac} is analogous, with the restriction $0<s<1/2$ coming from Lemma~\ref{lem:Mario}.
\end{proof}

\subsection{Properties of Bessel functions $I_k$}
\label{sec:preliminaries}

We collect in this subsection some properties of modified Bessel functions.
Let $I_k$ be the modified Bessel function of the first kind and order $k\in \Z$, defined as
\begin{equation}
\label{eq:Ik}
I_k(t) = i^{-k} J_k(it) = \sum_{m=0}^{\infty} \frac{1}{m!\,\Gamma(m+k+1)} \left(\frac{t}{2}\right)^{2m+k}.
\end{equation}
Since $k$ is an integer and $1/\Gamma(n)$ is taken to be equal zero if $n=0,-1,-2,\ldots$, the function $I_k$ is defined in the whole real line.
It is verified that
\begin{equation}
\label{eq:negPos}
I_{-k}(t) = I_k(t),
\end{equation}
for each $k \in \Z$. Besides, from~\eqref{eq:Ik} it is clear that
$I_0(0) = 1$ and $I_k(0) = 0$ for $k \ne 0$.
Also,
\begin{equation}
\label{eq:Ik>0}
I_k(t) \ge 0
\end{equation}
for every $k \in \Z$ and $t\ge0$, and
\[
\sum_{k \in \Z} e^{-2t} I_k(2t) = 1.
\]
From this it can be verified that
\begin{equation}
\label{eq:sumIk}
\sum_{k \in \Z^n} e^{-2nt} \prod_{i=1}^nI_{k_i}(2t) = 1.
\end{equation}
On the other hand, there exist constants $C,c>0$ such that
\begin{equation}
\label{eq:asymptotics-zero}
c t^k\le I_k(t)\le C t^k, \quad \text{ as } t\to0^+.
\end{equation}
In fact,
\begin{equation}
\label{eq:asymptotics-zero-ctes}
I_k(t)\sim \bigg(\frac{t}{2}\bigg)^k\frac{1}{\Gamma(k+1)}, \quad \text{ for a fixed } k\neq -1,-2,-3,\ldots~\text{and}~t\to0^+,
\end{equation}
see \cite{OlMax}.
It is well known (see \cite{Lebedev}) that
\begin{equation}
\label{eq:asymptotics-infinite}
I_k(t)=C e^t t^{-1/2}+ R_k(t),
\end{equation}
where
\[
|R_k(t)|\le C_k e^tt^{-3/2}, \quad \text{ as } t\to\infty.
\]
We also have (see \cite{OlMax}) that, as $\nu\to \infty$,
\begin{equation}
\label{eq:asymptotics-order-large}
I_{\nu}(z)\sim \frac{1}{\sqrt{2\pi \nu}}\bigg(\frac{e z}{2 \nu}\bigg)^{\nu}\sim \frac{z^\nu}{2^{\nu} \nu!}.
\end{equation}
For the following formula see \cite[p.~305]{Prudnikov2}. For $\Re c>0$, $-\Re \nu<\Re \alpha<1/2$,
\begin{equation}
\label{eq:rusos}
\int_0^{\infty} e^{-ct} I_{\nu}(ct) t^{\alpha-1}\,dt
= \frac{(2c)^{-\alpha}}{\sqrt{\pi}} \frac{\Gamma(1/2-\alpha)\Gamma(\alpha+\nu)
}{\Gamma(\nu+1-\alpha)}.
\end{equation}

\medskip

\noindent\textbf{Acknowledgments.}
This research was initiated when the second author was visiting The University of Texas at Austin.
She is grateful to the Department of Mathematics for their kind hospitality.



\end{document}